\documentclass[10pt,a4paper]{article}
\usepackage{amssymb,amsfonts,amsthm}
\usepackage{amssymb,amsfonts}
\usepackage{authblk}
\def\N{\mathbb{N}}
\def\R{\mathbb{R}}

\def\Z{\mathbb{Z}}

\def\PP{\mathcal{P}}
\def\QQ{\mathcal{Q}}
\def\CC{\mathcal{C}}
\def\UU{\mathcal{U}}
\def\VV{\mathcal{V}}
\def\mod{$\mbox{mod}$}
\newtheorem{thm}{Theorem}
\newtheorem{lem}{Lemma}

\newtheorem{pro}{Proposition}
\newtheorem{ques}{Question}

\newtheorem{de}{Definition}

\newtheorem{rem}{Remark}

\begin{document}



\title{Entropy rate of higher-dimensional cellular automata}



\author[1]{Fran\c cois Blanchard}
\affil[1]{ Laboratoire d'Analyse et de Math\'ematiques Appliqu\'ees, UMR 8050 (CNRS-U. de Marne-la-Vall\'ee), UMLV, 5 boulevard Descartes, 77454 Marne-la-Vall\'ee Cedex 2, France \thanks{E-mail address: \texttt{francois.blanchard@univ-mlv.fr}} }
\author[2]{Pierre Tisseur }
\affil[2]{ Centro de Matematica, Computa\c{c}\~ao e Cogni\c{c}\~ao, Universidade Federal do ABC, Santo Andr\'e, S\~ao Paulo, Brasil \thanks{E-mail address: \texttt{pierre.tisseur@ufabc.edu.br}}}
 
\date{\empty}

\maketitle
 
\begin{abstract} 
 We introduce the entropy rate of multidimensional cellular automata. This number  is invariant under shift--commuting isomorphisms; as opposed to the entropy of such CA, it is always finite.
The invariance property and the finiteness of the entropy rate result from basic results about the entropy of partitions of multidimensional cellular automata. We prove several results that show  that entropy rate of 2-dimensional automata preserve similar properties of the entropy of  one dimensional cellular automata.
 In particular we establish an inequality
 which involves the entropy rate, the radius of the cellular automaton and the entropy of the d-dimensional shift.
We also  compute the entropy rate of permutative bi--dimensional cellular automata and show that the finite value of the entropy rate
 (like the standard entropy of for one--dimensional CA) depends on the number of permutative sites.
 Finally  we define the topological entropy rate and prove that it is an invariant for topological shift-commuting conjugacy
and  establish some relations between topological and measure--theoretic entropy rates.
\end{abstract}

\section{Introduction}
A cellular automaton  (CA) is a continuous self-map $F$ on the configuration space
$A^{\Z^d}$, commuting with the group of shifts on this space. CA are simple computational devices for computer scientists and they are nice models for physicists. Mathematicians view them as an interesting family of topological and measurable dynamical systems. 

The entropy of a CA map $F$ acting on some full shift $A^{\Z^d}$, in its measure-theoretic as well as its topological versions ($h_\mu(A^{\Z^d},F)$ and $h(A^{\Z^d},F)$ respectively) is an important measure of the local unpredictability of the map. Each of the two entropies is an invariant under the suitable kind of conjugacy. 

The entropy of 1-dimensional CA is always finite. But when $d>1$ this measure is a crude one. Already in the two-dimensional case the entropy of a cellular automaton is often infinite. This is true for whole families of CA, the dynamics of which is especially tractable. For instance
it is shown in \cite{damico} that for the class of additive two-dimensional CA on $\{0,1\}^{\Z^2}$, which may be  seen as a subclass of 
two-dimensional permutative CA, defined  in Section \ref{permu}, the entropy is alway infinite. 
It was conjectured by Shereshevsky that for a two-dimensional CA the entropy could be 0 or infinite.
In \cite{Mey}  Meyerovitch has shown that there exist non-trivial examples of two-dimensional CA with finite positive entropy. 
To finish with the entropy of two-dimensional CA, we can say that it look impossible to establish some inequalities 
between the entropy of the automaton and the entropy of the group of shifts since for this last value we need to 
 divide by some square of the number of iterations (see definitions done by equality \ref{eq0} ).
 
 Here we introduce entropy rate for CA acting on $A^{\Z^2}$. It is not hard to obtain similar results for CA on $A^{\Z^d}$, $d > 2$, with proper changes in the definition of entropy rate.
It is derived from partial values of the entropy of the CA and can be expressed as follows for an $F$-invariant measure $\mu$ which is also invariant for the group of shifts:
$$ER_\mu(A^{\Z^2},F) =   \limsup_{n \to \infty} \frac{1}{n} h_\mu(\mathcal{S}_n,F),$$
where $\mathcal{S}_n$ is the clopen partition of $A^{\Z^2}$ according to the values of the coordinates in the square of side $2n+1$ centred at the origin. It is finite for any CA. It is very deeply grounded in the shift structure of the configuration space; as a consequence it is mostly significant when $\mu$ is also invariant under the group of shifts, and in this case it is an invariant for shift-commuting isomorphisms.
 Note that 
$\limsup_{n \to \infty} \frac{1}{n} h_\mu(\mathcal{S}_n,F)$ defined for all  
$F$-invariant measure $\mu$  is an invariant for continuous and shift-invariant isomorphisms only (see subsection \ref{subsectionERS0}).
The topological entropy rate  
$$ER(A^{\Z^2},F) =   \limsup_{n \to \infty} \frac{1}{n} h(\mathcal{S}_n,F)$$
has similar properties and similar limitations.

One could define the entropy rate of one-dimensional cellular automata: it is equal to their usual entropy, up to some multiplicative constant, and does not bring any further information about the dynamics. On the other hand, the entropy of a CA in higher dimensions is often infinite, whereas its entropy rate is always finite, like the entropy in one dimension, so entropy rate turns out to be more sensitive than entropy when $d\ge 2$. In particular, it makes it possible to obtain inequalities, as shown in Section \ref{major} and \ref{permu}.

Let $A$ be a finite set of cardinality  $\# A$. We denote by  $A^{\Z^d}$, the set of configurations or  maps 
from $\Z^d$ to $A$. In this paper we mainly restrict our study to the case $d=2$. 
We note that $A^{\Z^d}$ endowed with the product
topology of the discrete topologies on the sets $A$ is a compact space. 
Let $\Sigma$ be the group  generated by the the $d$ shifts $\sigma_j$  ($1\le j\le d$).  

Note that it is possible to generalize the Curtis-Hedlund-Lyndon theorem (see \cite{Hedlund}) 
and 
state  that for every cellular automaton $F$ there exists an integer $r$ 
called the radius of the CA and a
block map $f$ from $A^{(2r+1)^d}$ to $A$ such that 
$F\left(x(i_1,\ldots ,i_d)\right)=f(x([i_1-r,i_1+r],\ldots ,[i_d-r,i_d+r])$. 

{\bf Entropy}
The entropy (metrical ($h_\mu (T)$) or topological $h(T)$) is an isomorphism invariant that measures the 
complexity of the dynamical system $(X,\mu ,T)$ or $(X,T)$.
For each one-dimensional cellular automaton $F$ of radius $r$ it is well known that $h_\mu (F)\le h(F)\le 2r\ln (\# A)$.
In the ergodic setting (for the shift or the CA $F$) it was shown (see \cite{ti2000}) that
 $h_\mu (F)\le (\lambda^++\lambda^-) \cdot h_\mu (\sigma )\le 
2r\cdot h_\mu (\sigma )$ where $\sigma$ is the shift on $A^\Z$ and $\lambda^\pm$ are discrete Lyapunov exponents. 
In Proposition \ref{major1d} 
 we show that the last inequality  
$h_\mu (F)\le 2r\cdot h_\mu (\sigma)$ remains true for  shift and $F$-invariant measure $\mu$ for the one-dimensional case.
There exist some strong relations between dynamical properties of the CA like  equicontinuity and the fact  that  the entropy is equal to zero (see \cite{BT2000} and \cite{ti2009}).
Is there exists similar results for the entropy rate of two dimensional CA?
In the class of permutative one-dimensional CA the entropy rate is easy to compute. For instance when  $F$ is a CA of radius $r$ permutative in coordinates $-r$ and $r$ the value of the entropy is $h(F)=2r\times\ln (\# A)$.
 For two dimensional  permutative CA, the entropy $h_\mu (F)=+\infty$.

The Variational Principle (see for instance \cite{Wa}) which states that $h(F)=\sup_{\mu}h_\mu (F)$ implicitly introduces the question of the existence of a set of measures of maximum entropy: may it be empty? May it contain more than one measure? As far as we know those questions are open even when $d = 1$. Note that for the permutative class this set is not empty and contains the uniform measure. 

\medskip  

In this paper we introduce a formal definition of the entropy rate that is derived directly from 
the definition of the entropy. 
A first tentative and incomplete definition of measurable entropy rate was given by the second author in  \cite{ti2005} as a draft; 
a little later in \cite{lak}  Lakshtanov and Langvagen introduced some similar notions for the topological case. 
None of those two definitions allow to prove invariance under some class of isomorphisms.

\medskip

{\bf New definition and results}

In this paper we introduce the notion of entropy rate of partition $\mathcal{P}$ denoted by $ER_\mu (\mathcal{P}$,F) and define the measurable 
entropy rate $ER_\mu (A^{\Z^2},F)$ as the supremum over all the finite partitions of the entropy rate of a partition 
(see Definition \ref{ed1}, \ref{ed2} and \ref{ed3}).
Using some particular properties of the entropy of bi-dimensional cellular automata (see Lemma \ref{squareent}) we show in 
Proposition \ref{main} that there exists a partition 
$S_0$ such that $ER(A^{\Z^2},F)= ER_\mu (S_0,F)$ when $\mu$ is an $F$-invariant and shift commuting measure and establish in Proposition \ref{finite} that 
the the entropy rate is finite ($ER_\mu (A^{\Z^2},F)= ER_\mu (S_0,F)\le 8r\ln (\#A)$).

Next we show that for an $F$ and shift-invariant measure the entropy rate denoted by $ER_\mu (A^{\Z^2},F)$ is an invariant for the class of shift commuting isomorphism (see Proposition \ref{invar1}). In Subsection \ref{subsectionERS0} we prove that entropy rate of the partition $\mathcal{S}_0$:  $ER_\mu (S_0,F)$ is an invariant for continuous and 
shift-invariant isomorphism for all $F$-invariant measure $\mu$.

We also prove that for any CA $F: A^{\Z^2}\to A^{\Z^2}$ of radius $r$  permutative at the four sides of the square $E_r$ used to define the local rule $f$ (see Definition \ref{dpermu}) we can compute explicitly the entropy rate and obtain   
$ER_{\mu_\lambda} (A^{\Z^2},F)=8r\ln (\#A)$ where $\mu_\lambda$ is the uniform measure on $A^{\Z^2}$.
When there is less than 4 sides of the square $E_r$  with permutatives points we compute the entropy rate for the subclass 
of additive cellular automata and show that the entropy rate is proportional with the number of permutative points 
(see Proposition \ref{additiveCA}).
This result could be compared with the entropy of  additive one dimensional CA  where there is also a proportion between the entropy 
and the  number of permutative points (see \cite{damico}).

Moreover we also note that the uniform measure on $A^{\Z^2}$ is a measure of {\it maximum  entropy rate} for the 
classe of permutative CA whereas the uniform measure on $A^{\Z}$ is a measure of {\it maximum 
entropy} for permutative one-dimensional CA.
 More generaly we show in Theorem \ref{tmajor} that for any bi-dimensional cellular automaton $F$
  and  measure $\mu$ invariant by $F$ and by the group of shift $\Sigma$ on $A^{\Z^2}$ 
we have $ER_\mu (A^{\Z^2},F)\le 8r\cdot h_\mu (A^{\Z^2},\sigma )$ where $h_\mu (A^{\Z^2},\sigma)$ is the entropy of the two-dimensional shift. This result could be compared with the fact that $h_\mu (A^{\Z},F)\le 2r\cdot h_\mu (A^{\Z},\sigma)$ proved in Proposition  \ref{major1d} with the same setting for the measure.  
We note that the last inequality is optimal in a sense that it is an equality in the permutative case and 
that it is not possible to establish an analog one linking the entropy of the two dimensional shift and 
the entropy of the CA. Moreover the 
proof requires the use of many properties of the entropy and conditional entropy. 

\medskip
 
In Section \ref{topo} we introduce the topological entropy rate and show that like the measurable entropy rate,   
 it is finite (Proposition \ref{ertdef1}) and that it is an invariant for shift commuting homeomorphisms of $A^{\Z^2}$ 
 (Proposition \ref{invart}). 
Next we show  that for all positive integer $k\ge 1$ one has $ER(A^{\Z^2},F)=k\cdot ER(A^{\Z^2},F)$. This property is also 
shared by the entropy and the measurable entropy rate. 
Then we give a relation between the two entropy rate showing (see Proposition \ref{vp} ) that 
$$ER(A^{\Z^2},F)\ge \sup_{\mu\in M (F,\sigma )}\{ER_\mu (A^{\Z^2},F)\}
$$
 and 
$$ER(\mathcal{S}_0,F)\ge \sup_{\mu\in M (F)}\{ER_\mu (\mathcal{S}_0,F) $$ 
where $M(F)$ is the set of $F$-invariant measures and  $M(F,\sigma)$  the subset of $M(F)$ of  measures  invariant  for the group of shift $\Sigma$ on $A^{\Z^2}$.

\medskip

Another result shows (see  Proposition \ref{plongement}) that  topological entropy rate depends mainly on the local rule of the CA and not on  the dimension of the CA space. 
 More precisely when a  CA acts on a two-dimensional space but its block map can be reduced to a 
one-dimensional one, its topological entropy rate is equal (up to some multiplicative constant) to the  entropy of the corresponding one-dimensional CA. 


  All the presents results seem to show that 
entropy rate is a rather well extended notion of entropy for multi-dimensional cellular automata and could be used to 
 make progress in the understanding of these particular dynamical systems.
Some drawback could appear, for example the definition  use a  limit  superior 
($ER_\mu(\mathcal{S}_0,F) = \limsup_{n \to \infty} \frac{1}{n} h_\mu(\mathcal{S}_n,F)$)   instead of the entropy that appears like 
a simple limit. 
Nevertheless the entropy rate of permutative CA came from a limit (see Remark \ref{rem-converge} and Proposition \ref{permu-convergence}) and the values  
$\limsup_{n \to \infty} \frac{1}{n} h_\mu(\mathcal{S}_n,F)$ and $\liminf_{n \to \infty} \frac{1}{n} h_\mu(\mathcal{S}_n,F)$ differ 
only no maximum of a factor 8  (see Proposition \ref{compare-lim} and Proposition \ref{eq-mesure} (ii) for the topological case).
Moreover this last property   gives more 
meaning to the properties $ER_{\mu}(F)=0$ and $ER(F)=0$ that could be linked with some dynamical properties of the two dimensional  CA  
as it occurs  for the 
properties $h_\mu (F)=0$ and $h(F)=0$ (see for instance \cite{BT2000}, \cite{XT} and \cite{ti2009}).

Note that those results (for the topological and  measurable case) can easily be extended to dimensions higher than two using more complex notations.

\section{Definitions and background}\label{def}

\subsection{Symbolic spaces and  cellular automata}\label{defCA}

Let $A$ be a finite set or {\it alphabet}; its cardinality is denoted by $\# A$. 
For an integer $d\ge 1$ let  $A^{\Z^d}$ be the set of all maps $x\colon \Z^d \to A$; any such map 
 $x\in A^{\Z^d}$ is called a {\it configuration}. 
Given a finite subset $C$ of $\Z^d$, one defines a {\it pattern on} $C$ as a map $P \colon C \to A$, in other words, an element of $A^C$. When $d>1$ the usual concatenation of words can be extended to some patterns in the following way: given $C$, $C'\subset \Z^d$ such that $C\cap C' = \emptyset$ and two patterns, $P$ on  $C$ and $P'$ on $C'$, the pattern $P\bullet P'$ on $C\cup C'$ is the one such that $(P\bullet P')(z)= P(z) \hbox{ for } z \in C$ and $(P\bullet P')(z)= P'(z) \hbox{ for } z \in C'$. Again for $C\subset \Z^d$, the pattern $x_C$ is just the restriction of the map $x$ to the set of coordinates $C$.

The configuration space $A^{\Z^d}$ is endowed with the product
of the discrete topologies on the various coordinates.
For this topology $A^{\Z^d}$ is a compact metric
space. For $z=(i,j) \in \Z^2$ put $\vert z\vert = \sqrt{i^2+j^2}$; a metric compatible with this topology  is defined by the distance
$d(x,y)=2^{-h}$ where $h=\min\{\vert z\vert \,\mbox{ such that } x_z\ne y_z\}$.
The {\it shift} maps $\sigma^{i,j} \colon A^{\Z^d}\to A^{\Z^d},\ i,j\in \Z$ are defined by 
$\sigma^{i,j} (x)_{k,l}=(x_{k+i,l+j}),\ k, l\in \Z$. For $t \in \Z$ and $v = (i,j) \in \Z^2$ put $t.v = (ti,tj)$. The shift maps form a group. It is worth while to consider this {\it group of shifts} $\Sigma =\{\sigma^{i,j} |  i,j\in \Z\}$ as acting on $A^{\Z^d}$;
the dynamical system $(A^{\Z^d} ,\Sigma )$ is often called the {\it full shift} of dimension $d$. 

All probability measures $\mu$ on $A^{\Z^d}$ that we consider are defined on the Borel sigma-algebra $\mathcal{B}$ generated by the topology of $A^{\Z^d}$. 
 

The Curtis-Hedlund-Lyndon theorem 
states  that for every  cellular automaton $F$ there is a finite set $C\subset \Z^d$ and a map $f$ from the set of patterns on $C$ to $A$ such that for $z\in \Z^d$ one has $F(x)_z=f(x_{C+z})$; $f$ is called the {\it local map} of the CA $F$. One easily sees that equivalently there exist $r \in \N$, $E_r$ being the square centered at the origin of size $2r+1$ and a map $ f$ from the set of patterns on $E_r$ to $A$ with the same property. This is the form we are going to use. In this case the integer $r$ is called the radius of $F$.
Recall that the uniform measure on  $A^{\Z^d}$ is invariant under a cellular automaton $F$, i.e., $\mu\circ F = \mu$, if and only if $F$ is onto \cite{Hedlund}. 

\subsection{Entropy}\label{defent}
Given some probability space $(X,\mathcal{A},\mu)$ let $\mathbf{F}(X)$ be the set of all finite $\mathcal{A}$-measurable partitions of $X$.
If $\mathcal{P} =\{P_1,\ldots ,\,
P_n\}$ and $\mathcal{Q} =\{Q_1,\ldots ,\, Q_m\}$ are two measurable partitions of $X$, denote
by $ \mathcal{P} \vee \mathcal{Q}$ the partition $\{P_i\cap Q_j\, ; 1\le i \le n; \,\,
1\le j\le m\}$.
If for all $1\le i\le n$ there exists a subset $J\subset [1,\ldots,m]$ such that $P_i=\cup_{j\in J}Q_j$ we 
write that  $\mathcal{P}\curlyeqprec \mathcal{Q}$. 

Put $H_\mu (\mathcal{P} ) = \sum_{P\in\mathcal{P}}\mu (P)\log \mu (P)$. $H_\mu$ is {\it sub-additive}, that is, $H_\mu(\mathcal{P}\vee\mathcal{Q})\le H_\mu(\mathcal{P})+H_\mu(\mathcal{Q})$.
Whenever $\mathcal{P},\ \mathcal{Q} \in \mathbf{F}(X)$ and $\mathcal{P}\curlyeqprec \mathcal{Q}$ one has $H_\mu(\mathcal{P})\le H_\mu(\mathcal{Q})$.
By (\cite[Theorem 4.3]{Wa})
\begin{equation}\label{entrop}
\begin{array}{lll} 
(i)&H_\mu (\mathcal{P}\vee\mathcal{Q}/\mathcal{R})=H_\mu (\mathcal{P}/\mathcal{R})+H_\mu (\mathcal{Q}/\mathcal{P}\vee \mathcal{R}\le H_\mu (\mathcal{P}/\mathcal{R})+H_\mu (\mathcal{Q}/\vee \mathcal{R})\cr
(ii)&H_\mu (\mathcal{P}\vee\mathcal{Q})=H_\mu (\mathcal{P})+H_\mu (\mathcal{Q}/\mathcal{P})\le H_\mu (\mathcal{P})+H_\mu (\mathcal{Q}). \cr
\end{array}
\end{equation}

Let $T$ be a measurable transformation of $X$ leaving $\mu$ invariant: $\mu\circ T = \mu$.
The {\it entropy of the partition} $\PP$ with respect to $T$ is defined as $h_\mu (\PP,T) =
\lim_{n\to\infty}\frac{1}{n}H_\mu (\vee_{i=0}^{n-1}T^{-i}\PP )$. Remark that $h_\mu (\PP,T)$ is well-defined because by sub-additivity of $H_\mu$ the sequence $\frac{1}{n}H_\mu (\vee_{i=0}^{n-1}T^{-i}(\PP) )$ is non-increasing with $n$; in particular this implies that $h_\mu (\PP,T) \le H_\mu(\PP)$. Finally the {\it entropy of} $(X,T,\mu )$ is
$h_\mu(T) = \sup_{\PP \in \mathbf{F}(X)} h_\mu (\PP,T )$. 
Recall that $H_\mu (T^{-i}\PP)=H_\mu (\PP)$  and by \cite[Theorem 4.12]{Wa}
\begin{equation}\label{entineq}
h_\mu(\mathcal{Q},T) \le h_\mu(\mathcal{P},T)+H_\mu(\QQ|\PP).
\end{equation}
An {\it isomorphism} between two measure-theoretic dynamical systems $(X,\mathcal{A},\mu,T)$ and $(X',\mathcal{A}',\mu',T')$ is a 1-to-1, bi-measurable map $\varphi$ between two sets $E\in \mathcal{A}$ and $E'\in \mathcal{A}'$ such that $\mu(E) = \mu'(E') = 1$ and that $\varphi\circ T = T'\circ \varphi$ on the set $E$. When such a map exists $h_\mu(T) = h_{\mu'}(T')$, in other words the entropy is invariant under isomorphisms.

\medskip
Now for the topological setting. If $\UU,\ \VV$ are open covers of a compact space $X$ their join $\UU\vee \VV$ is the open cover consisting of all sets of the form $A\cap B$ 
where $A\in \UU$ and $B\in \VV$.
An open cover $\UU$ is {\it coarser} than an open cover $\VV$, or $\UU\curlyeqprec \VV$, if every element of $\VV$ is a subset of an element of $\UU$. If $\UU \preceq \VV$ and $\UU' \preceq \VV'$ then $\UU\vee \UU' \curlyeqprec \VV\vee \VV'$.

When $\UU$ is an open cover of $X$, put $H(\UU )=\ln (N(\UU))$, where $N(\UU )$ denotes the  smallest cardinality of a finite subcover of $\UU$. Like $H_\mu$ the function $H$ is sub-additive, in this case, $H(\mathcal{U}\vee\mathcal{V})\le H(\mathcal{U})+H(\mathcal{V})$. Whenever $\VV\curlyeqprec \UU$ one has $H(\VV)\le H(\UU)$.  

Let $T$ be a surjective continuous map of $X$. By sub-additivity of $H$ the sequence $\frac{1}{n}H (\vee_{i=0}^{n-1}T^{-i}(\UU) )$ is non-increasing with $n$; the {\it topological entropy} of the cover $\UU$ with respect to $T$ is defined as $h (\UU,T) =
\lim_{n\to\infty}\frac{1}{n}H (\vee_{i=0}^{n-1}T^{-i}(\UU) )$ and the {\it entropy} of $(X,T )$ is
$h(X,T) = \sup_\UU h(\UU,T)$ on the set $\mathbf{R}(A^{\Z^2})$ of all finite open covers of $X$. When $\UU$ is an open cover $h (\UU,T)\le H(\UU)$; when $\VV \preceq \UU$ are two open covers one has $h(\VV,T) \le h(\UU,T)$. Another important inequality is
\begin{equation}\label{entroptop}
h(\mathcal{U}\vee\mathcal{V},T)\le h(\mathcal{U},T)+h(\mathcal{V},T).
\end{equation}
Of course topological entropy is invariant under (topological) conjugacy, that is, if $\varphi\colon (X,T) \to (X',T')$ is a one-to-one continuous map such that $\varphi\circ T = T'\circ \varphi$, then $h(X,T) = h(X',T')$.


\section{Entropy rate for a measure}\label{edens}
Here we define the entropy rate of a cellular automaton $F$ for an $F$-invariant measure $\mu$. Then some of its basic properties are explored.

We first introduce two families of finite subsets of $\Z^2$ ($E_n$ was less formally introduced in the first Section):

\begin{de}\label{ed1}
$E_n\subset \Z^2$ is defined to be the square of size $2n+1$ centred at the origin: $E_n = \{v=(i,j)\in \Z^2\ |\ -n \le i, j\le n\}$. \\
For $n\ge r$, where $r$ is the radius of the CA, $E'_n$ is the outer band of width $r$ of $E_n$: $E'_n = E_n \setminus E_{n-r}$. 
\end{de}

To a finite measurable partition $\mathcal{P} \in \mathbf{F}(A^{\Z^2})$ one associates two other finite partitions with the help of $E_n$ and $E'_n$,:

\begin{de}\label{ed2}
For $\mathcal{P}\in \mathbf{F}(A^{\Z^2})$ one defines 
$$\mathcal{P}_n=\bigvee_{v \in E_n} \sigma^v(\mathcal{P})\ (\hbox{for } n \in \N)$$ and
$$\mathcal{P}'_n=\bigvee_{v \in E'_n} \sigma^v(\mathcal{P})\ (\hbox{for }n\ge r).$$
 \end{de}

 When setting $\mathcal{P}=\mathcal{S}_0$, where $\mathcal{S}_0$ is the clopen partition according to the value of the 0th coordinate, one has a particular expression for $(\mathcal{S}_0)_n$, which we denote by $\mathcal{S}_n$:
 $$\mathcal{S}_n=\bigvee_{v \in E_n} \sigma^v(\mathcal{S}_0)= (\{x\in A^{\Z^2}\ | \ x|_{E_n} = c\}\ | \ c\in A^{E_n}).$$
 Likewise put
 $$\mathcal{S}'_n=\bigvee_{v \in E'_n} \sigma^v(\mathcal{S}_0)= (\{x\in A^{\Z^2}\ | \ x|_{E'_n} = c\}\ | \ c\in A^{E'_n}).$$
 
The partitions $\mathcal{P}_n$ and $\mathcal{P}'_n$ have been introduced here in their general form for proving Propositions \ref{main} and \ref{invar1}. Apart from this technical use we do not understand their meaning well. In the sequel we use them mostly in one particular case, when $\mathcal{P}=\mathcal{S}_k$ or $\mathcal{S}'_k$ for some $k$; in this case they are clopen partitions according to local patterns, a classical tool in symbolic dynamics. 
 
Two properties of the partitions $\mathcal{S}_n,\ n \in \N$ do not hold for the partitions $\mathcal{S}'_n$: by the definitions $\mathcal({S}_i)_j = \mathcal{S}_{i+j}$; and the partitions $\mathcal{S}_n,\ n \in \N$ generate increasing algebras that converge to the Borel $\sigma$-algebra on $A^{\Z^2}$. The last property implies in particular that if $F$ is a CA and $\mu$ is an $F$-invariant measure on $A^{\Z^2}$ one has $h_\mu(A^{\Z^2},F) = \lim_{n \to \infty} h_\mu(\mathcal{S}_n, F)$ \cite{Wa}. As $\mathcal{S}_n$ is also an open cover of $A^{\Z^2}$, and since for any finite open cover $U$ there is $N$ such that $U\curlyeqprec S_N$, one also has $h(A^{\Z^2},F) = \lim_{n \to \infty} h(\mathcal{S}_n, F)$ \cite{Wa}.

\begin{de}\label{ed3}
Let $F$ be a cellular automaton on $A^{\Z^2}$ with radius $r$, and let $\mu$ be a probability measure on $A^{\Z^2}$, invariant under $F$. If $\mathcal{P}$ is a finite measurable partition of $A^{\Z^2}$, its entropy rate is
$$ER_\mu(\mathcal{P},F) = \limsup_{n \to \infty} \frac{1}{n} h_\mu(\mathcal{P}'_n,F);$$
the entropy rate of the dynamical system $(A^{\Z^2},F)$ endowed with the measure $\mu$ is the non-negative real number 
$$ER_\mu(A^{\Z^2},F) = \sup \{ER_\mu(\mathcal{P},F)\ |\ \mathcal{P} \in \mathbf{F}(A^{\Z^2})\}.$$
\end{de}

The first step for investigating entropy rate consists in remarking that entropy rate is the same for partitions $\mathcal{S}_n$ and $\mathcal{S}'_n$, and also the same for $\mathcal{S}_n$ and $\mathcal{S}_m$, $m\ne n$. 

\begin{lem}\label{squareent}
Let $F$ be a cellular automaton with radius $r$ acting on $A^{\Z^2}$, and $\mu$ be an $F$-invariant measure.\\
(i) Whenever $n\ge r$ one has
$$h_\mu(\mathcal{S}'_n,F) = h_\mu(\mathcal{S}_n,F),$$
\noindent (ii) for $n\ge r$ and $m\in\N$ one has
$$ER_\mu(\mathcal{S}_n,F)=ER_\mu(\mathcal{S'}_n,F)
\mbox{ and }ER_\mu (S_m)=ER_\mu(\mathcal{S}_0,F)$$ 
\end{lem}
\begin{proof}
(i) By the definition of entropy and since $\mathcal{S}_n = \mathcal{S}'_n \vee \mathcal{S}_{n-r}$,\begin{equation}\label{eq:def}
h_\mu(\mathcal{S}_n,F) = h_\mu(\mathcal{S}'_n\vee \mathcal{S}_{n-r},F) = \lim_{N\to\infty} \frac{1}{N} H_\mu\left(\bigvee_{i=0}^{N-1} F^{-i}(\mathcal{S}'_n) \bigvee_{i=0}^{N-1} F^{-i}(\mathcal{S}_{n-r})\right).
\end{equation} 
Because $F$ is a cellular automaton with radius $r$, the $v$th coordinate of $F(x)$, $v \in \Z^2$, is determined by all coordinates of $x$ that are within the square $E_r+v$. In particular all coordinates of $F(x)$ in $E_{n-r}$ are completely determined by the coordinates of $x$ in $E_n = E'_n \cup E_{n-r}$, that is to say, $\mathcal{S}_{n-r} \curlyeqprec F^{-1}(\mathcal{S}'_n \vee \mathcal{S}_{n-r})$ and more generally $F^{-i}(\mathcal{S}_{n-r})  \curlyeqprec F^{-i-1}(\mathcal{S}'_n \vee \mathcal{S}_{n-r})$. Applying $F^{-1}$ inductively and using this remark each time one gets  
$$\bigvee_{i=0}^{N-1} F^{-i}(\mathcal{S}'_n) \bigvee_{i=0}^{N-1} F^{-i}(\mathcal{S}_{n-r}) = \bigvee_{i=0}^{N-1} F^{-i}(\mathcal{S}'_n)\vee F^{-N+1}(\mathcal{S}_{n-r}).$$
Inject this simpler form into (\ref{eq:def}) and then apply (\ref{entrop}(ii)). This yields: 
$$h_\mu(\mathcal{S}_n,F) \le   \lim_{N\to\infty} \frac{1}{N} H_\mu\left(\bigvee_{i=0}^{N-1} F^{-i}(\mathcal{S}'_n)\right)+  \lim_{N\to\infty} \frac{1}{N} H_\mu\left(F^{-N+1}(\mathcal{S}_{n-r})\right),$$
hence
$$h_\mu(\mathcal{S}_n,F) \le  h_\mu(\mathcal{S}'_n,F)+  \lim_{N\to\infty} \frac{1}{N} H_\mu(F^{-N+1}(\mathcal{S}_{n-r})).$$
Now since $\mu$ is $F$-invariant the real number $H_\mu(F^{-N+1}(\mathcal{S}_{n-r})) = H_\mu(\mathcal{S}_{n-r})= K$ does not depend on $N$, so that in the end 
$$h_\mu(\mathcal{S}_n,F) \le  h_\mu(\mathcal{S}'_n,F)+  \lim_{N\to\infty} \frac{1}{N} K= h_\mu(\mathcal{S}'_n,F).$$
The reverse inequality is obvious since $\mathcal{S}'_n \curlyeqprec \mathcal{S}_n$. This establishes the first claim. 

(ii) Fix  $i \ge 0$: because of the obvious identity $(S_0)_i = S_i$ one has
$$ER_\mu(\mathcal{S}_i,F)=  \limsup_{n \to \infty} \frac{1}{n} h_\mu(S_{n+i},F) = \limsup_{n \to \infty} \frac{1}{n+i} h_\mu(S_{n+i},F) 
= ER_\mu(\mathcal{S}_0,F).$$
Using (i) the equality $ER_\mu(\mathcal{S}_m,F)=ER_\mu(\mathcal{S'}_m,F)$ immediately follows.
\end{proof}

With the help of this Lemma one shows that the entropy rate of the `square' partitions $\mathcal{S}_n$ is finite and does not depend on $n$. 

\begin{pro}\label{finite}
For any cellular automaton $F$ acting on $A^{\Z^2}$, any $F$-invariant measure $\mu$, any $i\ge 0$
$$ER_\mu(\mathcal{S}_i,F)=ER_\mu(\mathcal{S}_0,F)\le 8r\log(\# A)<\infty.$$
\end{pro}
\begin{proof}
By Lemma \ref{squareent}(i), since $\mathcal{S}'_n=\bigvee_{v \in E'_n} \sigma^v(\mathcal{S}_0)$, and by (\ref{entrop}(ii))
$$ER_\mu(\mathcal{S}_0,F) = \limsup_{n \to \infty} \frac{1}{n} h_\mu(\mathcal{S}'_n,F) = \limsup_{n \to \infty} \frac{1}{n} h_\mu(\bigvee_{v \in E'_n} (\sigma^v(\mathcal{S}_0,F))
$$
$$
\le \limsup_{n \to \infty} \frac{1}{n}\sum_{v \in E'_n} h_\mu(\sigma^v(\mathcal{S}_0),F).$$ 
Now as $\sigma^v(\mathcal{S}_0)$ is the partition according to the coordinate $v$, by elementary upper bounds one gets 
$$h_\mu(\sigma^v(\mathcal{S}_0),F) \le H_\mu(\sigma^v(\mathcal{S}_0)) \le \log(\# A).$$ 
Combined with the above upper bound for $ER_\mu(\mathcal{S}_0,F)$ and since the cardinality of $E'_n$ is less than or equal to $8rn$ this yields
$$ER_\mu(\mathcal{S}_0,F) \le \limsup_{n \to \infty} \frac{1}{n}\cdot 8rn\log(\# A) = 8r\log (\# A).$$
In view of Lemma \ref{squareent}(ii) this finishes the proof . 
\end{proof}

Of course this result would be false without the factor $\frac{1}{n}$ in the definition of $ER_\mu(\mathcal{P},F)$. 

In order to prove that entropy rate is a natural notion, one must make a new assumption: the measure $\mu$ should be shift-invariant, in addition to the previous requirement of being $F$-invariant. Call {\it bi-invariant} any measure that is invariant both under $F$ and under the group of shifts. 

\begin{pro}\label{main}
Let $\mu$ be a bi-invariant measure. For  any finite measurable partition $\mathcal{P}$ of $A^{\Z^2}$ one has
$$ER_\mu(\mathcal{P},F) \le ER_\mu(\mathcal{S}_0,F),$$
and therefore
$$ER_\mu(A^{\Z^2},F) = ER_\mu(\mathcal{S}_0,F).$$
\end{pro}
\begin{proof}
Given a finite measurable partition $\mathcal{P}$ fix some $\epsilon>0$. Since the partitions $\mathcal{S}'_n$ converge to the discrete partition as $n \to \infty$, the conditional entropy $H_\mu(\mathcal{P}|\mathcal{S}'_n)$ goes to 0 as  $n \to \infty$: choose $k$ such that $H_\mu(\mathcal{P}|\mathcal{S}'_k) \le \epsilon$.

From this inequality, keeping in mind that  $(\mathcal{S}'_k)_n=\mathcal{S}'_{n+k}$, one derives another one for $H_\mu(\mathcal{P}_n|\mathcal{S}'_{n+k})$ in the following way. 
By definition $\mathcal{P}_n=\bigvee_{v \in E_n} \sigma^v(\mathcal{P})$, so
$$H_\mu(\mathcal{P}_n|\mathcal{S}'_{n+k}) = H_\mu(\bigvee_{v \in E_n} \sigma^v(\mathcal{P})|\mathcal{S}'_{n+k}) \le \sum_{v \in E_n} H_\mu(\sigma^v(\mathcal{P})|\mathcal{S}'_{n+k}).$$
Note  that $\mathcal{S}'_{n+k}$ is a refinement of $\sigma^v(\mathcal{S}'_{k})$ for every $v\in E_n$, because the set $E_k +v \subset \Z^2$ is a subset of $E_{n+k}$. Thus  $H_\mu(\sigma^v(\mathcal{P})|\mathcal{S}'_{n+k}) \le H_\mu(\sigma^v(\mathcal{P})|\sigma^v(\mathcal{S}'_{k}))$. Due to the fact that $\mu$ is  invariant under the shifts, $H_\mu(\sigma^v(\mathcal{P})|\sigma^v(\mathcal{S}'_{k})) = H_\mu(\mathcal{P}|\mathcal{S}'_{k})$. Then it results from the former majoration of $H_\mu(\mathcal{P}_n|\mathcal{S}'_{n+k})$ that 
$$H_\mu(\mathcal{P}_n|\mathcal{S}'_{n+k})  \le \sum_{v \in E_n} H_\mu(\mathcal{P}|\mathcal{S}'_{k})\le 8rn\epsilon.$$

This, together with inequality \ref{entineq}, allows us to bound the dynamical entropy of $\mathcal{P}_n$ from above:
$$h_\mu(\mathcal{P}_n, F) \le h_\mu(\mathcal{P}_n\vee \mathcal{S}_{n+k}, F) \le h_\mu( \mathcal{S}_{n+k}, F) + H_\mu(\mathcal{P}_n| \mathcal{S}_{n+k})
$$
$$
\le h_\mu( \mathcal{S}_{n+k}, F)+8rn\epsilon,$$
hence 
$$ER_\mu(\mathcal{P},F)=\limsup_{n\to\infty}\frac{1}{n} h_\mu(\mathcal{P}_n, F) \le \limsup_{n\to\infty}\frac{1}{n} (h_\mu( \mathcal{S}'_{n+k}, F)+8rn\epsilon)
$$
$$
 = ER_\mu(\mathcal{S}'_k,F) + 8r\epsilon.
 $$
Letting $\epsilon$ go to 0 (or equivalently letting $k$ go to infinity) this implies that for any finite partition $\mathcal{P}$
$$ER_\mu(\mathcal{P},F) \le ER_\mu(\mathcal{S}'_k,F).$$
As was noted in Proposition \ref{finite}, $ER_\mu(\mathcal{S}'_k,F)=ER_\mu(\mathcal{S}_0,F)$ for any $k$. Since  $\mathcal{S}_0 \in \mathbf{F}(A^{\Z^2})$ one thus gets $ER_\mu(A^{\Z^2},F) =  ER_\mu(\mathcal{S}_0,F)$, which finishes the proof.
\end{proof} 
\begin{ques}
Is there exist a CA $F$, a $F$-invariant measure $\mu$ and a partition $\mathcal{P}$ such that 
$ER_\mu (\mathcal{P},F)>ER_\mu (\mathcal{S}_0,F)$? 
 \end{ques}
\begin{rem}
From the proof of the last Proposition every measure $\mu$ which satisfies the property 
$\forall \epsilon >0$ ,  $\exists k\in \N$ such that $\forall v\in \Z^2$  $H_\mu(\sigma^v(\mathcal{P})|\sigma^v(\mathcal{S}'_{k}))\le\epsilon$  for  $\mathcal{P}\in  \mathbf{F}(A^{\Z^2})$
 verifies   $ER_\mu(A^{\Z^2},F) =  ER_\mu(\mathcal{S}_0,F)$. This simple remark allows us to extend easily the 
set of probability measures $\mu$ where the entropy rate equals $ ER_\mu(\mathcal{S}_0,F)$. For instance each measure invariant for the group generated by  some 
iterations of the bi-dimensional shift also satisfies $ER_\mu(A^{\Z^2},F) =  ER_\mu(\mathcal{S}_0,F)$.

\end{rem}

The last result 
has two consequences. The first is straightforward: for a CA with radius $r$ on the alphabet $A$ and a bi-invariant measure $\mu$, the entropy rate of any finite partition is bounded by $8r\log (\# A)$, which is not as obvious as the corresponding coarse upper bound for entropy in the one-dimensional setting. The second is the following

\begin{pro}\label{invar1}
Let $(A^{\Z^2}, F, \mu)$ and  $(B^{\Z^2}, G, \nu)$ be two cellular automata 
endowed  with their 
respective bi-invariant measures. If there exists a measurable map $\varphi\colon A^{\Z^2} \to B^{\Z^2}$ such that 
\begin{enumerate}
\item{} $\varphi$ commutes with any shift,
\item{} $\varphi \circ F = G \circ \varphi$ and 
\item{} $\varphi\mu = \nu$, 
\end{enumerate} 
one has
$$ER_\mu(A^{\Z^2},F) \ge ER_\nu(B^{\Z^2},G).$$ In particular entropy rate is an invariant for the class of shift-commuting isomorphisms of CA. 
\end{pro}
\begin{proof}
The proof is an elementary application of the assumptions and of the previous results; we give it in some detail in order to show how it relies upon the various hypotheses on the isomorphism map $\phi$. 

 Lift any partition $\mathcal{P} \in \mathbf{F}(B^{\Z^2})$ into $\mathbf{F}(A^{\Z^2})$ by $\varphi^{-1}$: then\\
(1)  $H_\mu(\varphi^{-1}(\mathcal{P})) = H_\nu(\mathcal{P})$, since $\varphi\mu = \nu$; \\
(2) this, and the fact that $\varphi \circ G = F \circ \varphi$, imply that $h_\mu(\varphi^{-1}(\mathcal{P}),F) = h_\nu(\mathcal{P},G)$;\\
(3) one also has $\varphi^{-1}(\mathcal{P}'_{n})=(\varphi^{-1} (\mathcal{P}))'_{n}$ because $\varphi$ commutes with the group of shift. 

Applying (3) to $\mathcal{P}= \mathcal{S}_0(B^{\Z^2})$ and then (2), one gets for any $n>r$ 
$$h_\mu((\varphi^{-1}(\mathcal{S}_0(B^{\Z^2}))'_{n}),F) =h_\mu(\varphi^{-1}(\mathcal{S}'_{n}(B^{\Z^2})),F) 
= h_\nu(\mathcal{S}'_{n}(B^{\Z^2}),G).$$
Carried into the definitions of $ER_\mu$ and $ER_\nu$ this implies 
$$ER_\mu(\varphi^{-1}(\mathcal{S}_0(B^{\Z^2})),F)=\limsup_{n\to\infty}\frac{h_\mu ((\varphi^{-1}(\mathcal{S}_0(B^{\Z^2}))'_{n}),F)}{n}
$$
$$
=\limsup_{n\to\infty}\frac{h_\mu(\varphi^{-1}(\mathcal{S}'_{n}(B^{\Z^2})),F)}{n}=
$$
 
$$
=\limsup_{n\to\infty}\frac{h_\nu(\mathcal{S}'_{n}(B^{\Z^2}),G)}{n}
=ER_\mu(\varphi^{-1}(\mathcal{S}_0(B^{\Z^2})),F)= ER_\nu(\mathcal{S}_0(B^{\Z^2}),G);$$
$\varphi$ is a measurable map, so that $\varphi^{-1}(\mathcal{S}_0(B^{\Z^2})) \in \mathbf{F}(A^{\Z^2})$; taking this into account, we get  $ER_\mu(A^{\Z^2},F) \ge ER_\mu(\varphi^{-1}(\mathcal{S}_0(B^{\Z^2})),F)=ER_\nu(\mathcal{S}_0(B^{\Z^2}),G)$.
Since  $\nu$ is a bi-invariant measure from  Proposition \ref{main} we obtain 
$ER_\mu(A^{\Z^2},F) \ge ER_\nu(\mathcal{S}_0(B^{\Z^2}),G)=ER_\nu(B^{\Z^2},G)$.  


Finally when $\varphi$ is an isomorphism the inequality we just obtained applies in the two directions and the two entropy rate  are equal.
\end{proof}
It looks unlikely that one could obtain the same result after relaxing any of the invariance or commutation assumptions in this proposition. All of them are used somewhere.


The following result implies that if the measure is bi-invariant, the positivity of sequences of type $\left(\frac{1}{u_n}h_\mu (\mathcal{S}_{u_n},F)\right)_{n\in\N}$ is independent of the subsequence 
$u_n$ which shows that the definition of the entropy rate seems rather robust.  
\begin{pro}\label{compare-lim}
 For all two-dimensional cellular automata and bi-invariant measure $\mu$ one has:
$$
\limsup_{n\to\infty}\frac{h_\mu (\mathcal{S}_n, F)}{n}\le 8\times\liminf_{n\to\infty}\frac{h_\mu (\mathcal{S}_n, F)}{n}.
$$
\end{pro}
\begin{proof}

Roughly the proof use the fact that  a square $E'_{np}$ is a union of $8n$ squares of type $E'_p$ without its central part situated at more than 
$r$ coordinates of the near side of the big square. 
Let $(u_n)_{n\in\N}$ and $(v_n)_{n\in\N}$ two sequence of increasing positive integers such that 
$$
\liminf\frac{h_\mu (S_n,F)}{n}=\lim\frac{h_\mu (S_{u_n},F)}{u_n} \mbox{ and } 
\limsup\frac{h_\mu (S_n,F)}{n}=\lim\frac{h_\mu (S_{v_n},F)}{v_n}.
$$
Fix $p\ge r$ and $m\in\N$ such that $v_m\ge u_p$.  Putting $n=\lfloor \frac{v_m}{u_p}\rfloor$ we have 
$$(S_{u_p})'_{n+1}
=\bigvee_{i=-n}^{n}\sigma^{(n,i)}\mathcal{S}_{u_p}\bigvee_{i=-n}^{rn}\sigma^{(-n,i)}\mathcal{S}_{u_p}
\bigvee_{i=-n}^{n}\sigma^{(i,n)}\mathcal{S}_{u_p}\bigvee_{i=-n}^{n}\sigma^{(i,-n)}\mathcal{S}_{u_p}
$$
It follows  that 
$$
h_\mu \left((S_{u_p}\right)'_{n+1},F)
\le \sum_{i=-n}^{n} h_\mu (\sigma^{(n,i)}S_{u_p},F)+\sum_{i=-n}^{n} h_\mu (\sigma^{(-n,i)}S_{u_p},F) 
$$
$$
 +\sum_{i=-n}^{n} h_\mu (\sigma^{(i,n)}S_{u_p},F)+\sum_{i=-n}^{n} h_\mu (\sigma^{(i,-n)}S_{u_p},F)
$$
and using the shift invariance of $\mu$ we obtain $h_\mu ((S_{u_p})'_{n+1},F)\le 8n\cdot h_\mu (S_{u_p},F)$.
Since $(S_{u_p})'_{n+1}\succeq \mathcal{S}'_{v_m}$ we get 
$$
\frac{h_\mu (S'_{v_m},F)}{v_m}\cdot\frac{v_m}{(n+1)u_p}\le\frac{h_\mu ((S_{u_p})'_{n+1},F)}{(n+1)u_p}\le 8n\frac{h_\mu (S_{u_p},F)}{(n+1)u_p}.
$$
Letting $m\to \infty$ with $n=\lfloor \frac{v_m}{u_p}\rfloor$ and using  Lemma \ref{squareent} which say that $h_\mu (S'_n,F)=h_\mu (S_n,F)$  we obtain 
$$
\limsup\frac{h_\mu (S_n,F)}{n}=\lim\frac{h_\mu (S_{v_n},F)}{v_n}\le 8\cdot\frac{h_\mu (S_{u_p},F)}{u_p}.
$$ 
Since this last equality is true for all integer $p\ge r$ we can conclude writing 
$$
\limsup\frac{h_\mu (S_n,F)}{n}\le 8\cdot \lim_{p\to \infty}\frac{h_\mu (S_{u_p},F)}{u_p}=\liminf\frac{h_\mu (S_n,F)}{n}.
$$


\end{proof}
For all $\mu$-invariant map we have $h_\mu (T^k)=k\cdot h_\mu(T)$.
The next result show that entropy rate share this property.
\begin{pro}\label{puissances}
For all cellular automaton $F$ on $A^{\Z^2}$, $k\in\N$  and bi-invariant measure $\mu$  we have 
$ER_\mu (A^{\Z^2},F^k)=k\cdot ER_\mu((A^{\Z^2},F)$ . 
\end{pro}
\begin{proof}
From Proposition \ref{main} we only need to show that $ER_\mu (\mathcal{S}_0,F^k)=k\cdot ER_\mu (\mathcal{S}_0,F)$.
Since $F$ is a cellular automaton of radius $r$ we have $\vee_{i=0}^{k-1}F^{-i}(\mathcal{S}_0)  \curlyeqprec\mathcal{S}_{kr}$ 
and consequently   $\vee_{i=0}^{k-1}F^{-i}(\mathcal{S}_n)  \curlyeqprec \mathcal{S}_{n+kr}$.  
Hence 
$$
\limsup_{n\to\infty}\frac{h_\mu (\mathcal{S}_n,F^k)}{n}\le \limsup_{n\to\infty}\frac{h_\mu (\vee_{i=0}^{k-1}F^{-i}(\mathcal{S}_n) ,F^k)}{n}\le \limsup_{n\to\infty}\frac{h_\mu (\mathcal{S}_{n+kr},F^k)}{n}.
$$
Since   
$\limsup_{n\to\infty}\frac{h_\mu (\mathcal{S}_n,F^k)}{n}=\limsup_{n\to\infty}\frac{h_\mu (\mathcal{S}_{n+kr},F^k)}{n}=ER_\mu (\mathcal{S}_0,F^k)$ and 
since  $h_\mu (\vee_{i=0}^{k-1}F^{-i}(\mathcal{S}_n),F^k)=k\cdot h_\mu (\mathcal{S}_n,F)$  (see \cite{Wa}) we can conclude writing 
$$
ER(\mathcal{S}_0,F^k)=\limsup_{n\to\infty}\frac{h_\mu (\vee_{i=0}^{k-1}F^{-i}(\mathcal{S}_n) ,F^k)}{n}=k\cdot ER_\mu (\mathcal{S}_0,F).
$$.
\end{proof}
\begin{rem}\label{REM0}
More generally one can extend Lemma \ref{squareent}, Proposition \ref{main},  \ref{invar1} and  \ref{puissances} and obtain similar results for Proposition 
\ref{finite} and  \ref{compare-lim}  using the following definition for the $d$-dimensional 
case:
$$ER_\mu (F)=\sup\{ER_\mu(\mathcal{P},F)\vert \mathcal{P}\in F(A^{\Z^d})\}
$$
$$
 = \sup \{ \limsup_{n \to \infty} \frac{1}{n^{d-1}} h_\mu(\mathcal{P}'_n,F)\vert \mathcal{P}\in F(A^{\Z^d})\}.$$
In the $d$ dimensional case $\mathcal{P}'_n=\bigvee_{v\in  E_n^{'d}} \sigma^v \mathcal{P}$ where $E_{n}^{'d}$ is a $d$ dimensional empty hypercube of side $n$ and width $r$. 
\end{rem}

\subsection{Entropy rate with respect to the partition $\mathcal{S}_0$, an invariant for continuous and shift invariant isomorphism}\label{subsectionERS0}
$\mbox{  }$\\

 The following sequence of elementary results shows that for all $F$-invariant measure $\mu$ the entropy rate $ER_\mu (\mathcal{S}_0,F)$ share several properties (but not all) with  
$ER_\mu (A^{\Z^2},F)$  and is an invariant for  continuous and shift-invariant isomorphism. 
Note that we call cylinder any element of  a partition $\mathcal{S}_k$ ($k\in\N$).
\begin{de}\label{def-sbc}
A sliding bock code is a continuous map $\phi :A^{\Z^2}\to B^{\Z^2}$ 
 such that  any element $\sigma_A$ of the group of the shift on $A^{\Z^2}$ there 
exists $\sigma_B$ the corresponding element in the group of the shift on $B^{\Z^2}$ such that $\phi\circ \sigma_A=\sigma_B\circ\phi$.
\end{de}
 
\begin{pro}\label{ERS0}
For all bi-dimensional cellular automata $F$ and invariant measure $\mu$ one has  $ER_\mu (\mathcal{S}_0,F)=\sup_{CY(F)}\{ER_\mu (Q,F)\}$
where 
$CY(F)$ is the set of partitions by finite union of cylinders of $A^{\Z^2}$.
\end{pro}
\begin{proof}
From the definition of $CY(F)$, for all $\mathcal{Q}\in CY(F)$ there exists an integer $k\in\N$ such that $\mathcal{Q} \curlyeqprec \mathcal{S}_k$. 
It follows that for all $n\in\N$ we have $h_\mu (\mathcal{Q}_n,F)\le h_\mu (\mathcal{S}_{k+n},F)$ which implies that 
$ER_\mu (\mathcal{Q},F)\le ER_\mu (\mathcal{S}_k,F)=ER_\mu(\mathcal{S}_0,F)$ by Lemma \ref{squareent}.
\end{proof}

\begin{rem}
By the proof of Proposition \ref{ERS0} and Lemma \ref{squareent} it is straightforward  that 
 for all bi-dimensional cellular automata $F$ one has
$$
ER_\mu (\mathcal{S}_0,F)=\sup_{CY(F)}\left\{\limsup\frac{h_\mu (\mathcal{Q}_n,F)}{n}\right\}=\sup_{CY(F)}\left\{\limsup\frac{h_\mu (\mathcal{Q}'_n,F)}{n}\right\}.
$$
\end{rem}
The next result shows that $ER_\mu (\mathcal{S}_0,F)$ is a invariant for continuous and shift commuting isomorphism for each 
$F$-invariant probability measure $\mu$.

\begin{pro}\label{invar-ERS0}
Let $(A^{\Z^2}, F, \mu)$ and  $(B^{\Z^2}, G, \nu)$ be two cellular automata 
endowed  with their 
respective invariant measures. If there exists a sliding block code $\varphi\colon A^{\Z^2} \to B^{\Z^2}$ such that 
  $H_\mu(\varphi^{-1}(\mathcal{P})) = H_\nu(\mathcal{P})$,  $\varphi\mu = \nu$ and 
 $\varphi\mu = \nu$ 
then $ER_\mu \left(\mathcal{S}_0(A^{\Z^2}),F\right)=ER_\nu \left(\mathcal{S}_0(B^{\Z^2}),G\right)$.
\end{pro}

\begin{proof}
With these assumptions we can follow exactly the  proof of Proposition \ref{invar1} until the argument that $\varphi$ is a measurable map 
 (line 14 of the proof ) and substitute it by  
 $\varphi$ is a continuous and shift-invariant isomorphism or a one to one and onto sliding block code  from $A^{\Z^2}$ to $B^{\Z^2}$. 
Using simple compactness arguments (see the Curtis Hedlund Lindon Theorem \cite{Hedlund}) we can show that  $\varphi^{-1}(\mathcal{S}_0)\subset CY(F))$ which by 
  Proposition \ref{ERS0} implies that  
$$
ER_\mu \left(\mathcal{S}_0(B^{\Z^2}),F\right)\ge ER_\mu \left(\varphi^{-1}(\mathcal{S}_0)(B^{\Z^2}),F\right)=ER_\nu\left(\mathcal{S}_0(B^{\Z^2}),G\right).
$$
Finally since $\varphi$ is am isomorphism the inequality we obtained applies in the two directions and the two entropies rate are equal.
\end{proof}
Note that Proposition \ref{puissances} is clearly true  for $ER_\mu (\mathcal{S}_0,F)$  but Proposition \ref{compare-lim} that gives more meaning to the definition of the entropy rate using a limsup requires the shift invariance of the measure $\mu$.
If we compare $ER_\mu (A^{\Z^2},F)$ with $ER_\mu (\mathcal{S}_0,F)$   we can say that  $ER_\mu (A^{\Z^2},F)$  
 is significant when   $ER_\mu (A^{\Z^2},F)=ER_\mu (\mathcal{S}_0,F)$ which is mainly for shift-invariant measure and  $ER_\mu (\mathcal{S}_0,F)$ is only an invariant for 
continuous isomorphism.

\section{An upper bound for the entropy rate}\label{major}

The following basic result for one dimensional CA is similar to several inequalities  (that involved 
discrete Lyapunov exponents in \cite{ti2000} and \cite{ti2009})  in the ergodic setting  but is not written anywhere. 
\begin{pro}\label{major1d}
When $F$ is a one-dimensional cellular automaton of radius $r$ and  $\mu$ is a bi-invariant measure one has the  inequality     $h_\mu(A^\Z,F) \le 2r\cdot h_\mu(A^\Z,\sigma).$
\end{pro}
\begin{proof}
Let  $\alpha_0$ be the 
partition of $A^\Z$ by the central coordinate and $\alpha_p=\vee_{i=-n}^{n}\sigma^{-i} (\alpha_0)$.
Using the fact that $\mu$ is a $F$-invariant measure and $\lim_{p\to+\infty}\alpha_p$ is  the whole Borel $\sigma$-algebra $\mathcal{B}$  we obtain:
$$
h_\mu (F)=\lim_{p\to+\infty}h_\mu (F,\alpha_p)=\lim_{p\to+\infty}\lim_{n\to+\infty}\frac{H_\mu (\vee_{i=0}^{n-1}F^{-i}\alpha_p)}{n}
$$ 
Using the definition of a one-dimensional CA of radius $r$, for all $ p\in\N$  we can state that   $ \vee_{i=0}^{n-1}F^{-i}\alpha_p\curlyeqprec \vee_{i=-r(n-1)}^{r(n-1)}\sigma^{i}(\alpha_p) $ which 
implies that:
$$
h_\mu (F)=\lim_{p\to+\infty}h_\mu (F,\alpha_p)\le \lim_{p\to+\infty}\lim_{n\to+\infty}\frac{H_\mu \left( \vee_{i=-r(n-1)}^{r(n-1)}\sigma^{i}(\alpha_p) \right)}{n}.
$$  
 It follows that
$$
h_\mu (F)=\lim_{p\to+\infty}h_\mu (F,\alpha_p)\le  \lim_{p\to+\infty}\lim_{n\to+\infty}\frac{H_\mu \left( \vee_{i=-r(n-1)}^{r(n-1)}\sigma^{i}(\alpha_p) \right)}{2rn+1-2r}\cdot\frac{2rn+1-2r}{n}.
$$
Since $(\alpha_p)_{p\in\N}$  is a generating sequence for the transformation $\sigma$ and $\mu$ a shift-invariant measure we can state that for all $p\in\N$  
$$
h_\mu (\sigma)=h_\mu (\sigma ,\alpha_p )=\lim_{n\to+\infty}\frac{H_\mu \left( \vee_{i=-r(n-1)}^{r(n-1)}\sigma^{i}(\alpha_p) \right)}{2rn+1-2r}
$$ 
 which  allows us to  conclude.
\end{proof} 

The next results can be seen as a two-dimensional analogue of this
inequality, but its proof is not as simple. It is also a refinement of the coarse upper bound  in Proposition \ref{finite}.

The entropy $h_\mu(A^{\Z^2},\sigma)$ of the two-dimensional group of shifts for an invariant measure $\mu$ was introduced in \cite{Weiss}. As a function of  the partitions $\mathcal{S}_n$ one may write it as:
\begin{equation}\label{eq0}
h_\mu(A^{\Z^2},\sigma)=\lim_{n\to\infty}\frac{H_\mu \left(\vee_{v\in E_n}\sigma^v(\mathcal{S}_0)\right)}{(2n+1)^2} = \lim_{n\to\infty}\frac{H_\mu (\mathcal{S}_n)}{(2n+1)^2}.
\end{equation}

\begin{thm}\label{tmajor}
Let $F$ be a two-dimensional cellular automaton. 
If  $\mu$ is a bi-invariant measure on $A^{\Z^2}$ 
 one has:
$$
ER_\mu(A^{\Z^2},F)\le  8r \times h_\mu(A^{\Z^2},\sigma).
$$  
\end{thm}
\begin{proof}
First we claim that for any  $n\ge r$ the quantity $\frac{1}{p}H_\mu (\vee_{i=0}^{p-1}F^{-i}(\mathcal{S}'_n)\vert \mathcal{S}_{n-r})$ tends to a limit as $p \to \infty$ and that
\begin{equation}\label{condit}
\lim_{p\to\infty}\frac{H_\mu (\vee_{i=0}^{p-1}F^{-i}(\mathcal{S}'_n)\vert \mathcal{S}_{n-r})}{p}=h_\mu (\mathcal{S}'_n,F).
\end{equation}
Indeed since the equality  
$$
H_\mu (\mathcal{Q}\vert \mathcal{P})=H_\mu (\mathcal{P}\vee \mathcal{Q})-H_\mu (\mathcal{P})
$$
holds for all finite partitions $\mathcal{P}$,  $\mathcal{Q}$, one has
$$
H_\mu \left(\bigvee_{i=0}^{p-1}F^{-i}(\mathcal{S}'_n)\vert \mathcal{S}_{n-r}\right)=
H_\mu \left(\bigvee_{i=0}^{p-1}F^{-i}(\mathcal{S}'_n)\vee \mathcal{S}_{n-r}\right)-H_\mu (\mathcal{S}_{n-r}),
$$
which, taking into account the fact that  $\bigvee_{i=0}^{p-1}F^{-i}(\mathcal{S}'_n) \curlyeqprec \bigvee_{i=0}^{p-1}F^{-i}(\mathcal{S}'_n)\vee \mathcal{S}_{n-r} \curlyeqprec \bigvee_{i=0}^{p-1}F^{-i}(\mathcal{S}_n)$ and dividing by $p$, 
yields
$$
\frac{1}{p} (H_\mu(\bigvee_{i=0}^{p-1}F^{-i}(\mathcal{S}'_n))-H_\mu (\mathcal{S}_{n-r})) \le \frac{1}{p} H_\mu (\bigvee_{i=0}^{p-1}F^{-i}(\mathcal{S}'_n)\vert \mathcal{S}_{n-r})
$$
$$
\le\frac{1}{p} (H_\mu (\bigvee_{i=0}^{p-1}F^{-i}(\mathcal{S}_n))-H_\mu (\mathcal{S}_{n-r})).
$$
Passing to the limit as $p \to \infty$, the term $\frac{1}{p}H_\mu (\mathcal{S}_{n-r})$ vanishes, and the lower bound and the 
upper bound  converge to $h_\mu (\mathcal{S}'_n,F)$ and 
$h_\mu (\mathcal{S}_n,F)$ respectively. Those two quantities are equal by Lemma \ref{squareent}. One thus gets
$$
\lim_{p\to\infty}\frac{H_\mu (\vee_{i=0}^{p-1}F^{-i}(\mathcal{S}'_n)\vert \mathcal{S}_{n-r})}{p}=
h_\mu (\mathcal{S}_n,F)=h_\mu (\mathcal{S}'_n,F).
$$
 Equation \ref{condit} is proven. 
 
Applying (\ref{entrop}(i)) iteratively we note that 
$\left(H_\mu (\vee_{i=0}^{p-1}F^{-i}\mathcal{S}'_n\vert \mathcal{S}_{n-r})\right)_{p\in\N}$ is a subadditive sequence 
and it follows that   
  $$
\frac{h_\mu (\mathcal{S}'_n,F)}{n}=\lim_{p\to\infty}\frac{H_\mu (\vee_{i=0}^{p-1}F^{-i}\mathcal{S}'_n\vert \mathcal{S}_{n-r})}{pn}\le 
\frac{H_\mu (\vee_{i=0}^{\lfloor \sqrt{n}\rfloor-1}F^{-i}\mathcal{S}'_n\vert\mathcal{S}_{n-r})}{n\lfloor \sqrt{n}\rfloor}.
$$ 
 
 
where $\lfloor x\rfloor$ is the integer part of $x\in \R$. 
This implies that 
\begin{equation}\label{eq1}
ER_\mu(A^{\Z^2},F)\le\limsup_{n\to\infty}\frac{H_\mu
(\vee_{i=0}^{\lfloor \sqrt{n}\rfloor-1}F^{-i}(\mathcal{S}'_n)\vert\mathcal{S}_{n-r})}{n\lfloor \sqrt{n}\rfloor}.
\end{equation}

Observe that as $F$ is a cellular automaton of radius $r$, what happens from time 0 to time $\lfloor \sqrt{n}\rfloor-1$ in the square band $E'_n$ is completely determined by the coordinates in the square band $D_n=E'_{n+r\lfloor \sqrt{n}\rfloor}\setminus E'_{n-r\lfloor \sqrt{n}\rfloor}$ at time 0; in other words   
$$
\bigvee_{i=0}^{\lfloor \sqrt{n}\rfloor-1} F^{-i}(\mathcal{S}'_n)\curlyeqprec\bigvee_{v\in D_n}\sigma^v (\mathcal{S}_0).
$$
Recall that $\mathcal{S}_0$ is the partition according to the value of the $(0,0)$ coordinate.
 
The last inequation implies that 
$$
H_\mu \left(\vee_{i=0}^{\lfloor \sqrt{n}\rfloor-1} F^{-i}(\mathcal{S}'_n)\vert \mathcal{S}_{n-r}\right)
\le H_\mu \left(\vee_{v\in D_n}\sigma^v\mathcal{S}_0)\vert \mathcal{S}_{n-r}\right).
$$
Put $G_n=E_{n+r\lfloor \sqrt{n}\rfloor}\setminus E_{n-r}$, then $D_n = G_n \cup (E_{n-r} \setminus E_{n-r\lfloor \sqrt{n}\rfloor})$; injecting this into the latter entropy inequality one gets 
$$
H_\mu \left(\vee_{i=0}^{\lfloor \sqrt{n}\rfloor-1} F^{-i}(\mathcal{S}'_n)\vert \mathcal{S}_{n-r}\right)
\le H_\mu \left(\vee_{v\in G_n}\sigma^v\mathcal{S}_0\vert \mathcal{S}_{n-r}\right) +0,$$ hence obviously 
\begin{equation}\label{eq2}
H_\mu \left(\vee_{i=0}^{\lfloor \sqrt{n}\rfloor-1} F^{-i}(\mathcal{S}'_n)\vert \mathcal{S}_{n-r}\right)
\le H_\mu \left(\vee_{v\in G_n}\sigma^v\mathcal{S}_0\right).
\end{equation}

In order to obtain a convenient upper bound for the right-hand term in the last inequality, note that as a square band the set $G_n$ is the union of 4 rectangles of length $2(n+r\lfloor \sqrt{n}\rfloor)+1$ and of width $r(\lfloor \sqrt{n}\rfloor+1)$. 

Let $\chi(n) \in \{0,1\}$ and $k(n) \in \N$ be such that $r(\lfloor \sqrt{n}\rfloor+1)+\chi(n)$ is odd (when $r$ is even $\chi(n)=1$, but when $r$ is odd $\chi(n)$ varies with $n$) and $k(n) = \frac{1}{2}(r(\lfloor \sqrt{n}\rfloor+1)+\chi(n)-1)$.
Each of the four rectangles above is covered (not disjointly!) by at most
$\lfloor \frac{2n+1+r\lfloor\sqrt{n}\rfloor}{r(\lfloor\sqrt{n}\rfloor+1)+\chi(n)}\rfloor+1$ squares of size $r(\lfloor\sqrt{n}\rfloor+1)+\chi(n)$, each one of them a translate of the square $E_{k(n)}$.
A consequence is that the partition $\bigvee_{v\in G_n}\sigma^v\mathcal{S}_0$ is coarser than the supremum of the partitions generated by all coordinates belonging to at least one of  those squares of size $r(\lfloor\sqrt{n}\rfloor+1)+\chi(n)$. 

Since $\mu$ is preserved under the group of  shifts, for every $v \in \Z^2$ one has $H_\mu(\sigma^v(\mathcal{S}_{k(n)})$ 
$ = H_\mu(\mathcal{S}_{k(n)})$.
The inequality 
$$
H_\mu \left(\bigvee_{v\in G_n}\sigma^v\mathcal{S}_0\right)\le  4\left(\lfloor \frac{2n+1+r\lfloor\sqrt{n}\rfloor}{r(\lfloor\sqrt{n}\rfloor+1)+\chi(n)}\rfloor +1\right) \times H_\mu (\mathcal{S}_{k(n)}) 
$$
immediately follows, hence by a straightforward computation 
$$
\frac{1}{n\lfloor\sqrt{n}\rfloor}H_\mu \left(\bigvee_{v\in G_n}\sigma^v\mathcal{S}_0\right)\le \left(
\frac{8}{r(\lfloor \sqrt{n}\rfloor)^2}+\frac{1}{n}\ O(n)\right)\times H_\mu (\mathcal{S}_{k(n)})
$$
which, passing to the lim sup, implies that
\begin{equation}\label{eq3}
\limsup_{n\to\infty}\frac{1}{n\lfloor\sqrt{n}\rfloor}H_\mu \left(\bigvee_{v\in G_n}\sigma^v\mathcal{S}_0\right)\le 
\limsup_{n\to\infty}\left(
\frac{8}{rn}\right)\times H_\mu (\mathcal{S}_{k(n)}).
\end{equation}
By (\ref{eq0}), since $k(n) \to \infty$ as $n \to \infty$,
$$
h_\mu (A^{\Z^2},\sigma )=\lim_{n\to\infty}\frac{H_\mu (\mathcal{S}_{k(n)})}{(2k(n) +1)^2} ;
$$
replacing $k(n)$ by its value and carrying the latter inequality into (\ref{eq3}) one gets
$$
\limsup_{n\to\infty}\frac{1}{n\lfloor\sqrt{n}\rfloor}H_\mu \left(\bigvee_{v\in G_n}\sigma^v\mathcal{S}_0\right)\le 
\limsup_{n\to\infty}
\frac{8}{rn}(r(\lfloor \sqrt{n}\rfloor+1)+\chi(n))^2 \frac{H_\mu (\mathcal{S}_{k(n)})}{(2k(n)+1)^2}
$$
$$
= 8r\times h_\mu (A^{\Z^2},\sigma).
$$
The desired inequality follows by (\ref{eq1}) and (\ref{eq2}):
$$
ER_\mu(A^{\Z^2},F)\le 8r\times h_\mu (A^{\Z^2},\sigma).
$$
\end{proof}

\begin{rem}\label{REM1}
Using the more general definition for a $d$-dimensional CA given in Remark \ref{REM0}
 one can easily extend the last proof  (using more complex notations)
 and  show that 
$$
ER_\mu ( A^{\Z^d},F)\le (\partial S_r )\times h_\mu (A^{\Z^d},\sigma )
$$
where $\partial S_r$ is the surface of the hypercube of side $2r+1$. 
\end{rem} 
\section{Permutative CA}\label{permu}

Computing the entropy $h_\mu (F)$ of a cellular automaton is a difficult work in general. 
For some examples it is possible to show that $h_\mu (F)=0$ (for instance see \cite{ti2009} or \cite{XT}).
Another exception is the permutative and additive case where (see  \cite{damico}) exact computation is much easier (tractable). 
In the multidimensional case, the entropy of permutative CA is not finite \cite{damico}.
In this section we show how to compute the exact value of their 
entropy rate in the two-dimensional case. 

Recall that patterns and the concatenation $P \bullet P'$ of disjoint patterns $P$ and $P'$ are defined in Subsection \ref{defCA}.

\begin{de}\label{dpermu}
Let $F$ be a two-dimensional cellular automaton of radius $r$, let
$f:A^{(2r+1)^2}\to A$ be the block map defining $F$. Choose a position $(i,j)\in E_r$.
The cellular automaton $F$ is called permutative at $(i,j)$ if
for any $a\in A$, any given pattern $P$ on $E_r \setminus \{(i,j)\}$
there exists $b\in A$ such that $f(P\bullet ( \{(i,j)\} \to b))=a$.
\end{de}
In other words, $F$ is permutative at $(i,j)$ if, given a pattern $P$ on $E_r \setminus \{(i,j)\}$ and some letter $a$, one can choose a letter $b$ such that $a$ is the output of $f$ for $P$ completed by $b$ at $(i,j)$.


Let $\mu_\lambda$ be the uniform measure on $A^{\Z^2}$.
By definition  for all  finite set $E\in\Z^2$ 
and element $a\in \mathcal{S}_0$ 
one has  $\mu_\lambda (\bigcap_{v\in E}\sigma^{-v}a)=A^{-(\#E)}$. Since $\# \mathcal{S}_0=\#A$ and 
each element of $(\bigvee_{v\in E}\sigma^{-v}\mathcal{S}_0)$ have the same measure we have  
(see \cite{Wa})
\begin{equation}\label{m-uniforme}
H_{\mu_\lambda}(\bigvee_{v\in E}\sigma^{-v}\mathcal{S}_0)=\#E\cdot\log (\# A).
\end{equation}
 
Denote by $Pt=\{p_1,p_2,p_3,p_4\}$ the set of points situated at the centers of the four sides of the square 
$E_r$ ($p_1=(0,r)$, 
$p_2=(0,-r)$, $p_3=(-r,0)$ and $p_4=(r,0)$ ).
\begin{pro}\label{permu1}
If a cellular automata $F$ of radius $r$ is permutative at all the points in $Pt$  the entropy of $F$ with 
respect to the uniform measure  $\mu_\lambda$ is 
$$
ER_{\mu_\lambda} (A^{\Z^2},F)= 8\times r\times \log( \#A).
$$ 
\end{pro}
\begin{proof}

We note that $\mu_\lambda$ is a $F$-invariant measure since 
it was shown by Winston in \cite{Wi}  that a cellular automaton permutative at only one point $p\in E_r$ is  invariant by the uniform measure.
Since $\mu_\lambda$ is invariant with respect to the group of shift by Proposition \ref{main} we have $ER_{\mu_\lambda} (A^{\Z^2},F)=ER_{\mu_\lambda} (\mathcal{S}_0,F)$.  
Hence using Proposition \ref{finite} 
 we can finish the proof showing that  for all $p\in\N$ and $n\in\N$ one has 
$$
ER_{\mu_\lambda} (A^{\Z^2},F)\ge 8r\times \log(\#A ).
$$

For all $1\le s\le 4$ and $p\in\N$ let $R_s^p$ $(1\le s\le 4)$ be the four sides of the square 
$E'_p=\cup_{s=1}^{4}R_s^p$.
More formally  $R_1^p=\{(i,j)\in\Z^2\vert p-r\le i\le p\mbox{ and }-p\le j\le p\}$, 
$R_2^p=\{(i,j)\in\Z^2\vert -p\le i\le p-1\mbox{ and }p-r\le j\le p\}$, $R_3^p=\{(i,j)\in\Z^2\vert -p+r\le i\le -p\mbox{ and }-p\le j\le p-1\}$ and $R_4^p=\{(i,j)\in\Z^2\vert -p+1\le i\le p-1\mbox{ and }-p+r\le j\le p\}$. 
 For all $1\le s\le 4$ define $\mathcal{R}_s^p=\bigvee_{v\in R_s^p}\sigma^{-i}\mathcal{S}_0$.
Using the commutativity of $F$ and $\sigma$ we can write that  for all positive integer $n$ one has  
$$
\bigvee_{i=0}^{n-1} F^{-i}(\mathcal{S}'_p) =
 \bigvee_{i=0}^{n-1} F^{-i} \left(\bigvee_{s=1}^{4}\mathcal{R}_s^p\right)=
 \bigvee_{s=1}^{4}\left(\vee_{i=0}^{n-1} F^{-i}(\mathcal{R}_s^p)\right).
$$
Since $F$ is permutative at $(r,0)$ one has
$$
\bigvee_{i=0}^{n-1}F^{-i}(\mathcal{R}_1^p)\curlyeqsucc \bigvee_{i=0}^{(n-1)r}\sigma^{(-i,0)}(\mathcal{R}_1^p).
$$
 
More generaly since $F$ is permutative at $(r,0)$, $(0,r)$, $(-r,0)$ and $(0,-r)$  one has 
$$
\bigvee_{i=0}^{n-1} F^{-i}(\mathcal{S}'_p)
\curlyeqsucc \bigvee_{i=0}^{(n-1)r}\sigma^{(-i,0)}(\mathcal{R}_1^p)\bigvee_{i=0}^{(n-1)r}\sigma^{(0,-i)}(\mathcal{R}_2^p)
 \bigvee_{i=0}^{(n-1)r}\sigma^{(i,0)}(\mathcal{R}_3^p) \bigvee_{i=0}^{(n-1)r}\sigma^{(0,i)}(\mathcal{R}_4^p)
$$
$$
=\bigvee_{s=1}^4 \left(\bigvee_{R_s(n,p)}\sigma^{v}\mathcal{S}_0\right)
$$
where $R_1(n,p)=\Z^2\cap\{-p+1\le j\le p\mbox{ and }p-r\le i\le p+(n-1)r\}$, $R_2(n,p)=
\Z^2\cap\{-p+1\le i\le p-1\mbox{ and }p-r\le j\le p+(n-1)r\}$, $R_3(n,p)=\Z^2\cap\{-p\le j\le p\mbox{ and }-p-(n-1)r\le i\le -p+r\}$ and 
$R_4(n,p)=\Z^2\cap\{-p+1\le i\le p\mbox{ and }-p-(n-1)r\le j\le -p+r\}$.

From equality \ref{m-uniforme} we get 
\begin{equation}\label{convergencemu}
H_{\mu_\lambda} (\vee_{i=0}^{n-1} F^{-i}\mathcal{S}'_p)\ge 
H_{\mu_\lambda}\left(\bigvee_{s=1}^4\left(\bigvee_{R_s(n,p)}\sigma^{v}\mathcal{S}_0\right)\right)=8rnp \log (\# A)
\end{equation}
which finish the proof.
\end{proof}

Using the same techniques of proof it is possible to compute 
the entropy rate of  CA permutative at $(i,r)$, $(j,-r)$, $(-r,k)$, $(r,l)$ with $(-r< i,j,k,l< r)$.


\begin{rem}\label{rem-converge}
Note that if $F$ is permutative at all points in $Pt$ equation \ref{convergencemu}  implies that   
$\left(\frac{1}{n}h_\mu (\mathcal{S}_n,F)\right)_{n\in\N}$ is a converging sequence. 
Recall that 
Proposition \ref{compare-lim} brings some basic informations about this sequence and 
the following result show that the limit also exists for CA permutative at only two 
points in $Pt\subset E_r$. 
\end{rem}

\begin{pro}\label{permu-convergence}
If $F$ is a CA of  radius $r$  permutative at $(0,r)$ and $(-r,0)$ or at 
$(r,0)$ and $(0,-r)$  then 
$\left(\frac{1}{n}h_{\mu_\lambda }(\mathcal{S}_n,F)\right)_{n\in\N}$ 
is a converging sub-additive sequence.
\end{pro}
\begin{proof}

Since $F$ is  permutative at  $(0,r)$ and $(-r,0)$ or at 
$(r,0)$ and $(0,-r)$ we 
  get for all  
$m\ge \lceil\frac{p}{r}\rceil$  
$$
\mathcal{S}_{n+p} \curlyeqprec \vee_{i=0}^{m-1} F^{-i}\left(\sigma^{(p,p)}\mathcal{S}_n\right) \vee \sigma^{(-p,-p)}\mathcal{S}_p.
$$ 
Since $\mu_\lambda$ is a shift invariant measure we get  
$$
H_\mu (\mathcal{S}_{n+p})\le H_\mu (\vee_{i=0}^{m-1}F^{-i}\mathcal{S}_n)+H_\mu (\mathcal{S}_p)
$$
and for all $k\in\N$ we have 
$$
H_\mu (\vee_{i=0}^{k-1}\mathcal{S}_{n+p})\le H_\mu (\vee_{i=0}^{k+m-2}F^{-i}\mathcal{S}_n)+H_\mu (\vee_{i=0}^{k-1}F^{-i} \mathcal{S}_p)
$$
which, dividing by $k$ and then letting $k$ go to $\infty$, implies that $h_\mu (\mathcal{S}_{n+p},F)\le h_\mu (\mathcal{S}_{n},F)+h_\mu (\mathcal{S}_{p},F)$.
 
It follows that  $\left(\frac{1}{n}h_\mu (\mathcal{S}_n,F)\right)_{n\in\N}$ is a non-increasing sub-additive  converging sequence 
\end{proof}
  
When a CA is not permutative at the four sides of the square $E_r$ the calculous of the entropy rate is more complicated.
Nevertheless for the subclass of additive CA like for the one dimensional case for the entropy (see \cite{damico})  we can compute explicitly the value of $ER_{\mu_\lambda}(A^{\Z^2},F)$ and show that the entropy is proportional to the number 
of ''additive sites''.
Note that an  additive CA is cellular automaton defined thanks to an additive local rule. 
To simplify the notations we will restrict our results to the space $\{0,1\}^{\Z^2}$.  

Call $F_{(1,2)}$ the CA defined thanks the local rule $f_{(1,2)}(E_r)=x_{(r,0)}+x_{(0,r)} \mod\, 2$, 
$F_{(3,4)}$ the CA defined by $f_{(3,4)}(E_r)=x_{(-r,0)}+x_{(0,-r)} \mod\, 2$,  
$F_{(1,3)}$ the CA defined by $f_{(1,2)}(E_r)=x_{(r,0)}+x_{(0,r)} \mod\, 2$ and 
$F_{(1)}$ the CA defined by $f_{(1)}(E_r)=x_{(r,0)} \mod\, 2$.
\begin{pro}\label{additiveCA}
We have  $ER_{\mu_\lambda}\left(\{0,1\}^{\Z^2}, F_{(1,2)}\right)=ER_{\mu_\lambda}\left(\{0,1\}^{\Z^2}, F_{(3,4)}\right)$ 
 $=ER_{\mu_\lambda}\left(\{0,1\}^{\Z^2}, F_{(1,3)}\right)=4r\ln (2)$ and $ER_{\mu_\lambda}\left(\{0,1\}^{\Z^2}, F_{(1)}\right)=2r\ln (2)$. 
\end{pro}
\begin{proof}

We first show that $ER_{\mu_\lambda}\left(\{0,1\}^{\Z^2},F_{(1,2)}\right)=4r\ln (2)$.
We use the same notation than in the proof of Proposition \ref{permu1} where $(R_s^p)$  $(1\le s\le 4)$ represent the  four sides of the empty square $E'_p$ and     
$\mathcal{R}_s^p=\bigvee_{v\in R_s^p}\sigma^v (\mathcal{S}_0)$.
Since    
 $\mathcal{S}'_p=\vee_{s=1}^4\mathcal{R}_s^p$.
it is easily seen that 
$$
\bigvee_{i=0}^{n-1} F_{(1,2)}^{-i}(\mathcal{S}'_p) \curlyeqsucc
 \bigvee_{i=0}^{n-1} F_{(1,2)}^{-i} (\bigvee_{s=3}^{4}\mathcal{R}_s^p )=
 \bigvee_{s=3}^{4}\left(\bigvee_{i=0}^{n-1} F_{(1,2)}^{-i}(\mathcal{R}_s^p) \right)
$$
$$
 \curlyeqsucc \bigvee_{i=0}^{(n-1)r}\sigma^{(-i,0)}(\mathcal{R}_1^p) \bigvee_{i=0}^{(n-1)r}\sigma^{(0,-i)}(\mathcal{R}_2^p)
=\bigvee_{v\in R_1(n,p)\cup R_2(n,p)}\sigma^v (\mathcal{S}_0).
$$
Using Equality \ref{m-uniforme} we obtain $H_\mu (\vee_{i=0}^{n-1} F^{-i}_{(1,2)}\mathcal{S}_p)\ge 4rnp\log (\#2)$ which implies that $ER_{\mu_\lambda}\left(\{0,1\}^{\Z^2}, F_{(1,2)}\right)\ge  4r\ln (2)$.
 To obtain the reverse inequality note that   since $F$ is permutative at $(r,0)$ and $(0,r)$  we get for all $k\in\N$
$$
\bigvee_{i=0}^{2k+1}F_{(1,2)}^{-i}(\mathcal{R}_3^k \bigvee\mathcal{R}_4^k) \curlyeqsucc \mathcal{S}_k.
$$ 
From Lemma \ref{squareent} and basic properties of the entropy (see \cite{Wa}) we can assert that $\forall n\in \N$  
$$
h_\mu\left(\mathcal{S}_n,F_{(1,2)}\right)\le h_\mu\left(\bigvee_{i=0}^{2n+1}F_{(1,2)}^{-i}(\mathcal{R}_3^n \bigvee\mathcal{R}_4^n),F_{(1,2)}\right)
=h_\mu\left(\mathcal{R}_3^n \bigvee\mathcal{R}_4^n,F_{(1,2)}\right).
$$ 


  Following the same argument than in the proof of Proposition \ref{finite} we get 
  $$
\hskip -7 cm ER_{\mu_\lambda}\left(\{0,1\}^{\Z^2}, F_{(1,2)}\right)
$$
$$
\le
   \limsup_{n \to \infty} \frac{1}{n} h_\mu(\mathcal{R}_3^n \bigvee\mathcal{R}_4^n,F_{(1,2)}) = \limsup_{n \to \infty} \frac{1}{n} h_\mu(\bigvee_{v \in  R_3^n\cup R_4^n} (\sigma^v(\mathcal{S}_0,F_{(1,2)}))
   $$
   $$
  \le \limsup_{n \to \infty} \frac{1}{n}\sum_{v \in R_3^n\cup R_4^n} h_\mu(\sigma^v(\mathcal{S}_0),F_{(1,2)})
$$
$$
\le 
   \limsup_{n \to \infty} \frac{1}{n}\sum_{v \in  R_3^n\cup R_4^n} H_\mu (\sigma^v(\mathcal{S}_0))\le 4r\ln (2)
  $$
 
 which finally  bring that  $ER_{\mu_\lambda}\left(\{0,1\}^{\Z^2}, F_{(1,2)}\right)=4r\ln (2)$.
  Using the same arguments for $F_{(3,4)}$  and $F_{(1,3)}$ we obtain  
$$
ER_{\mu_\lambda}\left(\{0,1\}^{\Z^2}, F_{(3,4)}\right)=
ER_{\mu_\lambda}\left(\{0,1\}^{\Z^2}, F_{(1,3)}\right)=4r\ln (2).
$$
  Using only the  side $R_1^p$ of the empty squares $E'_p$  it is easily seen that 
$$
ER_{\mu_\lambda}\left(\{0,1\}^{\Z^2}, F_{(1)}\right)=2r\ln (2).
$$ 
 \end{proof}




\section{Topological Entropy rate}\label{topo}
Here we introduce entropy rate in the topological setting. Recall that  relevant properties of the entropy function and of topological entropy are given in Subsection \ref{defent}.

Denote by $\mathbf{R}(A^{\Z^2})$ the set of all finite open covers of $A^{\Z^2}$. In the same way as for partitions in Section \ref{edens}, for $\CC \in \mathbf{R}(A^{\Z^2})$ put $\mathcal{C}'_n=\bigvee_{v \in E'_n} \sigma^v(\mathcal{C})\ (\hbox{for }n\ge r)$ and $\mathcal{C}_n=\bigvee_{v \in E_n} \sigma^v(\mathcal{C})$. Recall that the partitions $\mathcal{S}_n$ and $\mathcal{S}'_n$ introduced in the same Section are also open covers of the set $A^{\Z^2}$.
\begin{de}
Let $F$ be a cellular automaton on $A^{\Z^2}$ with radius $r$. 
The {\it entropy rate} of $\mathcal{C}\in \mathbf{R}(A^{\Z^2})$ 
 is defined as
$$
ER(\mathcal{C},F) = \limsup_{n \to \infty} \frac{1}{n} h(\mathcal{C'}_n,F);
$$
The entropy rate of the topological dynamical system $(A^{\Z^2},F)$  is the non-negative real number 
$$
ER(A^{\Z^2},F) = \sup_{\mathcal{C}\in\mathbf{R}(A^{\Z^2})} \{ER(\mathcal{C},F)\}.
$$
\end{de}

\subsection{First results about topological entropy rate}
\begin{lem}\label{edmaj}
Let $\UU$, $\VV$ be two open covers of $A^{\Z^2}$ with $\UU  \curlyeqprec \VV$. Then
$ER(\UU,F) \le ER(\VV,F)$.
\end{lem}
\begin{proof}
For $v \in \Z^2$, owing to the fact that $\sigma^v$ is a homeomorphism $\sigma^v(\UU)$ and  $\sigma^v(\VV)$ are also open covers of $A^{\Z^2}$ and $\sigma^v(\UU)  \curlyeqprec \sigma^v(\VV)$. For $n\in \N$ it follows that $\UU'_n  \curlyeqprec \VV'_n$ and this implies that $h(\UU'_n,F)\le h(\VV'_n,F)$, hence $ER(\UU,F) \le ER(\VV,F)$.
\end{proof}

The next result is a topological analogue of Lemma \ref{squareent} together with Proposition \ref{finite}. The proofs are similar. 
\begin{pro}\label{ertdef1}
For any cellular automaton $F$ of radius $r$ acting on $A^{\Z^2}$, any integer $n\ge r$ one has 
$h(\mathcal{S}_n,F)=h(\mathcal{S}'_n,F)$ and for any $n\in\N$ one has $ER(\mathcal{S}_n,F)=ER(\mathcal{S}_0,F)$. Moreover for all $k\in\N$  we have 
$
ER(\mathcal{S}_k,F)=ER(\mathcal{S}_0,F)\le 8r\log (\# A).
$
\end{pro}
\begin{proof}
Since the topological entropy function $H$ has the same sub-additivity property as $H_\mu$, and 
since for every finite open cover $\CC$ one has $h(\CC,F)\le H(\CC)$ (see section 2), 
we can use the same arguments as in the proofs of Lemma \ref{squareent} and Proposition 3.5 
 to obtain this result. 
\end{proof}
 The next result is the topological analogue of Proposition \ref{main}.  
\begin{pro}\label{idented}
For any cellular automaton $F$ on $A^{\Z^2}$ one has
 $ER(A^{\Z^2},F)=ER(\mathcal{S}_0,F)$.
\end{pro}
\begin{proof}
The common diameter of the elements of $\mathcal{S}_k$ goes to 0 as $k \to \infty$. By the Lebesgue Covering Lemma, for any cover $\mathcal{C}\in \mathbf{R}(A^{\Z^2})$ there exists a positive  integer $k$ such that 
$\mathcal{C} \curlyeqprec \mathcal{\mathcal{S}}_k$. By Lemma \ref{edmaj} it follows that  $ER(\CC,F) \le ER(\mathcal{S}_k,F)$. By Proposition \ref{ertdef1} $ER(\mathcal{S}_k,F)= ER(\mathcal{S}_0,F)$, which means that any open cover $\CC$ has entropy rate less than or equal to $ER(\mathcal{S}_0,F)$. Since $\mathcal{S}_0 \in \mathbf{R}(A^{\Z^2})$ the result follows. 
\end{proof}
The next result shows that since it has the same properties it is possible to choose another definition for the entropy rate of an open cover $\mathcal{C}$:
$\overline{ER(\mathcal{C},F)} = \limsup_{n \to \infty} \frac{1}{n} h(\mathcal{C}_n,F)$.
 Note that since  for any $n\in\N$ $\mathcal{C}_n\curlyeqsucc \mathcal{C}'_n$ 
$$
\overline{ER(\mathcal{C},F)} = \limsup_{n \to \infty} \frac{1}{n} h(\mathcal{C}_n,F)\ge \limsup_{n \to \infty} \frac{1}{n} h(\mathcal{C}'_n,F)=ER(\mathcal{C},F).
$$
Note that we have chosen  $ER(\mathcal{C},F)$ for its similarity with the measurable case.

\begin{pro}\label{secondtype}
For any cellular automaton $F$ on $A^{\Z^2}$ one has 

$$
\sup_{\mathcal{C}\in\mathbf{R}(A^{\Z^2})} \{\overline{ER(\mathcal{C},F)}\}=\sup_{\mathcal{C}\in\mathbf{R}(A^{\Z^2})} \{ER(\mathcal{C},F)\}=ER(\mathcal{S}_0,F)=ER(A^{\Z^2},F).
$$

\end{pro}
\begin{proof}
 Following the arguments of the proof of  Lemma \ref{edmaj}  we can assert that for any open covers $\mathcal{V}\curlyeqsucc\mathcal{U}$ one has 
$h(\mathcal{U}_n,F)\le h(\mathcal{V}_n,F)$. Since for any cover $\mathcal{C}\in \mathbf{R}(A^{\Z^2})$ there exists a positive  integer $k$ such that 
$\mathcal{C} \curlyeqprec \mathcal{\mathcal{S}}_k$ it follows that 
$\limsup_{n \to \infty} \frac{1}{n} h(\mathcal{C}_n,F)$ $\le \limsup_{n \to \infty} \frac{1}{n} h(\mathcal{S}_{n+k},F)
=  \limsup_{n \to \infty} \frac{1}{n} h(\mathcal{S}'_{n+k},F)$ $=ER(\mathcal{S}_0,F)$ $= ER(A^{\Z^2},F)$ from Propositions \ref{ertdef1} and \ref{idented}.
\end{proof} 


\begin{ques}
Is it possible to obtain similar results of Proposition \ref{secondtype} 
 for some class of non trivial measure in the measurable case? 
\end{ques}

Recall that  a sliding block is a continuous map from $A^{\Z^2}\to B^{\Z^2}$  that commute with all shifts.




\begin{pro}\label{invart}
Let $(A^{\Z^2},F)$ and $(B^{\Z^2},G)$ be two cellular automata. If there exists a sliding block code   
$\varphi: A^{\Z^2}\to B^{\Z^2}$  
such that $\varphi\circ F=G\circ\varphi$ then 
$$
ER(A^{\Z^2},F)\ge ER(B^{\Z^2},G).
$$
 More particularly topological entropy rate
is an invariant for the class of bijective sliding block codes $\varphi :A^{\Z^2}\to B^{\Z^2}$ .
\end{pro}
\begin{proof}

 Since $\varphi\circ F=G\circ\varphi$ for all $n\in\N$ and open cover $\mathcal{C}$ we have 
$$
h(\varphi^{-1}(\mathcal{C}_n(B^{\Z^2})),F)=h(\mathcal{C}_n(B^{\Z^2}),G).
$$
Moreover since $\varphi$ commute with the group of shift 
one has $\left[\varphi^{-1}(\mathcal{S}_0)\right]_n=\varphi^{-1}[(\mathcal{S}_0)_n]$. Using   Proposition \ref{ertdef1} we get
$$
ER(A^{\Z^2},F)\ge ER\left(\varphi^{-1}\left(\mathcal{S}_0(B^{\Z^2})\right) ,F\right)=
\limsup_{n\to\infty}\frac{h\left(\left[\varphi^{-1} (\mathcal{S}_0(B^{\Z^2})\right]_n,F\right)}{n}
$$

$$
=\limsup_{n\to\infty}\frac{h\left(\varphi^{-1}\left[ (\mathcal{S}_0(B^{\Z^2})_n\right] ,F\right)}{n}
=\limsup_{n\to\infty}\frac{h(\mathcal{S}_0(B^{\Z^2})_n ,G) }{n} 
=ER(\mathcal{S}_0(B^{\Z^2}),G).
$$
 Using Proposition \ref{idented} which states that $ER(B^{\Z^2},G)=ER(\mathcal{S}_0(B^{\Z^2}),G)$ 
 we can conclude.
  \end{proof}

The first part of the following results  shows that entropy rate exhibit similar properties than entropy 
and the second gives more meaning to the definition of the entropy rate and to the  property $ER(A^{\Z^2},F)=0$.
\begin{pro}\label{eq-mesure}
$\mbox{  }$\\
(i)
For all cellular automaton $F$ on $A^{\Z^2}$ and positive integer $k$ we have 
$$ER(A^{\Z^2},F^k)=k\cdot ER(A^{\Z^2},F).$$
(ii) For all two-dimensional cellular automata one has:
$$
\limsup_{n\to\infty}\frac{h (\mathcal{S}_n, F)}{n}\le 8\times\liminf_{n\to\infty}\frac{h (\mathcal{S}_n, F)}{n}.
$$
\end{pro}
\begin{proof}
(i) Similar to the proof of Proposition \ref{puissances}.
\end{proof}
\begin{proof}
(ii) Since by definition $F$  commute with the group of shift 
$h(\sigma^v(\mathcal{C}),F)=h(\mathcal{C},F)$ for any open cover $\mathcal{C}$.
Using this equality we can follows the same proof than for the measurable case for 
shift invariant measure 
(see Proposition \ref{compare-lim}).
\end{proof}

 \subsection{Relation between topological and measurable entropy rate}

\begin{pro}\label{vp}
Let $F$ be a cellular automaton from $A^{\Z^2}\to A^{\Z^2}$. Then  
$$
ER(A^{\Z^2},F)\ge \sup_{\mu\in M (F,\sigma  )}\{ER_\mu (A^{\Z^2},F)\} \mbox{ and  }
ER(A^{\Z^2},F)\ge \sup_{\mu\in M (F)}\{ER_\mu (\mathcal{S}_0,F)\}
$$
where $M(F)$ is the set of $F$-invariant measures and $M(F,\sigma )$ the set of all bi-invariant measures.
\end{pro}

\begin{proof}
Since each set $S_n$ ($n\in\N$) is a partition and also an open cover, 
 the lowest cardinality of any finite subcover of $S_n$ is equal to the cardinality of the 
finite partition $S_n$ ($N(S_n)=\# (S_n)$).

Since for all finite partition $\alpha$ and measure $\mu$ one has $H_\mu (\alpha )\le \log (\# \alpha )$  (see \cite{Wa})   we can assert that 
  for all integer  
$p\ge 0$ for all $F$-invariant measure one has     
$$
H_\mu (\vee_{i=0}^{p-1} F^{-i}\mathcal{S}_n)\le \log \left( N(\vee_{i=0}^{p-1} F^{-i}\mathcal{S}_n)\right)
$$
which implies that for all $n\in\N$ we have  $h_\mu (\mathcal{S}_n,F)\le h(\mathcal{S}_n,F)$ and 
allow us to  state the following inequality 
$$
ER_\mu (S_0,F)\le \limsup\frac{h(\mathcal{S}_n,F)}{n}=ER(\mathcal{S}_0,F) 
$$
that 
prove the second statement of this Proposition.
From Theorem \ref{main} and Proposition \ref{idented} one has $ER_\mu (A^{\Z^2},F)=ER_\mu (\mathcal{S}_0,F)$ and $ER(\mathcal{S}_0,F)= ER(A^{\Z^2},F)$ which allows to conclude.
\end{proof}


It seems not clear if in general  there exists some variational principle 
between the topological entropy rate $ER(A^{\Z^2},F)=ER(\mathcal{S}_0,F)$ and $ER_\mu (\mathcal{S}_0,F)$ (not $ER_\mu (A^{\Z^2}$
 $ ,F)$) 
because in order to show that $ER_\mu (\mathcal{S}_0,F)\ge ER(\mathcal{S}_0,F)$ we can note use classical 
arguments  of the standard variational principle's proof. For instance using some arguments  of standard proof of the variational principle 
(see \cite{Wa})  we can 
show that  given any open cover $\beta$  there exist  a   finite partition $\xi$ and measure $\mu$ 
such that  $h_\mu (\xi, F)\ge h(\beta ,F )$. This can not implies that $ER(\mathcal{S}_0,F)\ge ER_\mu (\mathcal{S}_0,F)$.

 
We believe that the quantity $ER(A^{\Z^2},F)-\sup_{\mu\in M(F)}\{ER_\mu (\mathcal{S}_0, F)\}$ represents some  rate 
of none scale invariance dynamic for the multi-dimensional cellular automaton. Note that this value is equal to zero for permutative CA (see Proposition \ref{max-ER}).

\medskip 
\begin{de}
If an $F$-invariant measure $\mu$ verifies $ ER(A^{\Z^2},F)= ER_\mu(A^{\Z^2},F)$ we say 
that $\mu$ is a maximum entropy rate measure.
\end{de}
\begin{pro}\label{max-ER}
The  uniform measure on $A^{\Z^2}$ is a measure of maximum rate entropy 
  for all  bi-dimensional cellular automata $F$  permutative at the points $(0,r)$, $(0,-r)$. $(-r,0)$ and $(r,0)$.
\end{pro}
\begin{proof}
The proof is straightforward from Proposition  \ref{permu1},  \ref{ertdef1} and \ref{idented}. 
\end{proof}
Note that the uniform measure  is a measure of maximum entropy for one dimensional  permutative CA.

\begin{ques}
In \cite{Mey} Meyerovitch shows that there exist bi-dimensional CA with finite and positive entropy.
We wonder if there exists some  bi-dimensional CA such that $h(F)=\infty$ and $ER(F)=0$.
\end{ques}
 
\subsection{Entropy rate and CA'extensions }

In the following we give another rather basic argument that underline that the notion of entropy rate 
 could be better than entropy to  quantify  the complexity  of  multidimensional CA.  Recall that in dimension one 
the entropy rate is equal to the entropy up to a multiplicative constant.

In Subsection \ref{defCA} we remind the reader that for any cellular automaton $F$ there exists a unique associated block map that defines it completely.  
In the following we show that when a  CA acts on a two-dimensional space but its block map can be reduced to a 
one-dimensional one, its entropy rate is equal (up to some multiplicative constant) to the  entropy of the corresponding one-dimensional CA.

\begin{de}
If $F$ is a one-dimensional CA with corresponding block map $f:A^{2r+1}\to A$; the extension of $F$ to dimension 2 is the two-dimensional CA $\overline{F}$ defined by the local map 
$\overline{f}:A^{(2r+1)^2}\to A$ such that for 
any pattern $P$ on $E_r$ one has $\overline{f}(P)=f(p)$, where $p = P_{(0,-r)}P_{(0,-r+1)}\ldots P_{(0,r)}$.
\end{de}
In other words the local map $\overline{f}$, instead of reading all the coordinates in the square $E_r$, only reads those of the form $(0,i),\ -r\le i\le r$, and its action is that of $f$ on those coordinates.

\begin{pro}\label{plongement}
If $F$ is a one  dimensional CA and $\overline{F}$ is its extension to dimension 2 one has 
$$
 ER(A^{\Z^2},\overline{F})=2\cdot h(A^{\Z},F). 
$$
\end{pro}
\begin{proof}
From Proposition \ref{ertdef1} one has 
$$
ER(A^{\Z^2},\overline{F})=ER(\mathcal{S}_0,\overline{F})=
\limsup_{n\to\infty}\frac{h(,C_n,\overline{F})}{n}
$$
 with $\CC_0=\mathcal{S}_0$. 
For all $n\in\N$ define  $\alpha_n=\vee_{i=-n}^{n}\sigma^{(i,0)}\mathcal{S}_0$ and note that   the open cover (which is also a partition) 
$\CC_n=\mathcal{S}_n=\vee_{j=-n}^{n}\sigma^{(0,j)}\alpha_n$.  
Using the shift commutativity of $F$   we obtain   for all $k\in\N$
$$
 H(\vee_{i=0}^{k-1} \overline{F}^{-i}\CC_n)= H(\vee_{i=0}^{k-1} \overline{F}^{-i}(\vee_{j=-n}^{n}\sigma^{(0,j)}\alpha_n))=H(\vee_{j=-n}^{n} \sigma^{(0,j)}(\vee_{i=0}^{k-1} \overline{F}^{-i}\alpha_n)).
$$

 From the definition of $\overline{F}$ we can assert that for all $i\in\N$ one has  $\overline{F}^{(-i)}(\alpha_n)=F^{(-i)} \alpha_n\curlyeqprec \alpha_{n+ri}$  where 
$r$ is the radius of the CA $F$. Since for all $j\in\Z-\{0\}$ one has $\sigma^{(0,j)}\alpha_n\perp \alpha_n$  it follows that $\forall k\in\N$ one has 
 $$
 H(\vee_{i=0}^{k-1} \overline{F}^{-i}\CC_n)
=H(\vee_{j=-n}^{n} \sigma^{(0,j)}(\vee_{i=0}^{k-1} F^{-i}\alpha_n))
 =(2n+1)H(\vee_{i=0}^{k-1}F^{-i}\alpha_n)
 $$
  which implies  that $h(\CC_n,\overline{F})=(2n+1)h(\alpha_n,F)$. 
  Since $(\alpha_n)_{n\in\N}$ is a generating sequence $\lim_{n\to\infty}h(\alpha_n,F)$ $=h(F)$ which allows us  to  conclude.
\end{proof} 

\begin{rem}
When the dimension $d>1$ we can use a more general definition (likewise that given in  Remark \ref{REM0} for the measurable case )  to extend  Proposition  
\ref{plongement} and show that : 
$$
 ER(A^{\Z^d}, \overline{F})=2^{d-1}\cdot h(A^{\Z},F) 
$$
where $\overline{F}$ is the extension in dimension $d$ of the one-dimensional CA $F$.
\end{rem}



\begin{thebibliography} {999}

\bibitem{BT2000} F. Blanchard, P. Tisseur, Some properties of cellular automata with
equicontinuity points, Ann. Inst. Henri Poincaré, Probabilités et Statistiques {\bf36}(5) (2000), 569-582.

\bibitem{XT} X. Bressaud, P. Tisseur, On a zero speed sensitive cellular automaton,  Nonlinearity {\bf 20} (2007), 1-19.

\bibitem{damico} M. D'amico, G. Manzini and L. Margara. On computing the
entropy of cellular automata. Theoret. Comput. Sci., {\bf 290}(3) (2003), 1629-1646.

\bibitem{Hedlund}G.A. Hedlund, Endomorphisms and automorphisms of the shift dynamical system, Math. Syst. Theory {\bf 3} (1969) 320-375. 

\bibitem{lak} E. L. Lakshtanov and E. S. Langvagen, Entropy of multidimensional cellular automata, 
Problems of Information Transmission {\bf 42}(1) (2005), 38-45.

\bibitem{Mey} T. Meyerovitch, Finite entropy for multidimensional cellular automata. Ergodic
Theory Dynam. Systems {\bf 28}(4) (2008), 1243-1260.


\bibitem{sher} Mark A. Shereshevsky. Expansiveness, entropy and polynomial growth for groups
acting on subshifts by automorphisms. Indag. Math. (N.S.),{\bf 4}(2)  (1993),  203-210.

\bibitem{ti2000} P. Tisseur, Cellular automata and Lyapunov exponents, Nonlinearity {\bf 13} (2000) 1547-1560. 

\bibitem{ti2005}
P. Tisseur,  Always Finite Entropy 
 and Lyapunov exponents of two-dimensional cellular automata, arXiv:math/0502440 (2005).

\bibitem{ti2009}P. Tisseur, Density of periodic points, invariant measures and almost
equicontinuous points of cellular automata, Advances in Applied Mathematics {\bf 42} (2009),  504-518.

\bibitem{Weiss}  Y. Katznelson and B. Weiss, Commuting measure-preserving transformations
	 {\bf 12}(2)  (1972), 161-173.
     
\bibitem{Wa} 
    P. Walters,
    An introduction to ergodic theory. Springer,
    Berlin-Heidelberg-New York, 1982.

\bibitem{Wi}  
     S. J. Willson, On the ergodic theory of cellular automata, 
Theory of Computing Systems Volume {\bf 9}(2) (1975), 132-141.     

\end{thebibliography}
\end{document}